\documentclass[10pt]{article}
\usepackage{color}
\usepackage{amssymb}
\usepackage{amsthm,array,amssymb,amscd,amsfonts,latexsym, url}
\usepackage{amsmath}

\newtheorem{theo}{Theorem}[section]
\newtheorem{prop}[theo]{Proposition}

\newtheorem{lemm}[theo]{Lemma}
\newtheorem{coro}[theo]{Corollary}
\newtheorem{rema}[theo]{Remark}
\newtheorem{Definition}[theo]{Definition}

\title{On the universal ${\rm CH}_0$ group of  cubic hypersurfaces}
\author{Claire Voisin
\\CNRS, Institut de Math\'ematiques de Jussieu}
\date{}

\newfont{\gothic}{eufb10}
\begin{document}
\maketitle

\begin{abstract} We study the existence of a Chow-theoretic decomposition of the diagonal of a smooth cubic hypersurface,
or equivalently, the universal triviality of its ${\rm CH}_0$-group.
We prove that for  odd dimensional cubic hypersurfaces or for cubic
fourfolds, this is equivalent to the existence of a cohomological
decomposition of the diagonal, and we translate geometrically this
last condition. For cubic threefolds $X$, this turns out to be
equivalent to the algebraicity of the minimal class $\theta^4/4!$ of
the intermediate Jacobian $J(X)$. In dimension $4$, we show that a
special cubic fourfold with discriminant not divisible by $4$ has
universally trivial ${\rm CH}_0$ group.

 \end{abstract}
\section{Introduction}
Let $X$ be a smooth rationally connected projective variety over
$\mathbb{C}$. Then ${\rm CH}_0(X)=\mathbb{Z}$, as all points of $X$
are rationally equivalent. However, if $L$ is a field containing
$\mathbb{C}$, e.g. a function field, the group ${\rm CH}_0(X_L)$ can
be different from $\mathbb{Z}$. As explained in \cite{ACTP}, the
group ${\rm CH}_0(X_L)$ is equal to $\mathbb{Z}$ for any field $L$
containing $\mathbb{C}$ if and only if, for $L=\mathbb{C}(X)$, the
diagonal (or generic) point $\delta_L$ is rationally equivalent over
$L$ to a constant point $x_L$  for some (in fact any) point $x\in
X(\mathbb{C})$. Following \cite{ACTP}, we will then say that $X$ has
universally trivial ${\rm CH}_0$ group. Observe that, on the other
hand, the equality
$$\delta_L=x_L\,\,{\rm in}\,\,{\rm CH}_0(X_L)={\rm CH}^n(X_L),\,\,n={\rm dim}\,X$$
is, by the localization exact sequence applied to Zariski open sets
of $X\times X$ of the form $U\times X$, equivalent to the vanishing
in ${\rm CH}^n(U\times X)$ of the restriction of $\Delta_X-X\times
x$, where $U$ is a sufficiently small dense Zariski open set of $X$
and $\Delta_X\subset X\times X$ is the diagonal of $X$. This
provides  a Bloch-Srinivas decomposition of the diagonal
\begin{eqnarray}
\label{eqchodec}\Delta_X=X\times x+Z\,\,{\rm in}\,\,{\rm CH}^n(X\times X),
\end{eqnarray}
where $Z$ is supported on $D\times X$ for some proper closed subset
$D$ of $X$. As in \cite{voisinjag}, we will call an equality
(\ref{eqchodec}) a {\it Chow-theoretic decomposition of the
diagonal}. So, having universally trivial ${\rm CH}_0$ group is
equivalent to admitting a Chow-theoretic decomposition of the
diagonal, but the second viewpoint is much more geometric, and leads
to the study of weakened properties, like the existence of  a {\it
cohomological decomposition of the diagonal}, which is the
cohomological counterpart of (\ref{eqchodec}),  studied for
threefolds in \cite{voisinjag} :
\begin{eqnarray}
\label{eqcohdec}[\Delta_X]=[X\times x]+[Z]\,\,{\rm in}\,\,H^{2n}(X\times X,\mathbb{Z}),
\end{eqnarray}
where $Z$ is supported on $D\times X$ for some proper closed
algebraic subset $D$ of $X$. In these notions, integral coefficients
are essential in order to make the property restrictive, as the
existence of decompositions as above with rational coefficients
already follows from the assumption that ${\rm CH}_0(X)=\mathbb{Z}$
(see \cite{blochsrinivas}).

 Projective space has
universally trivial ${\rm CH}_0$ group. It follows that rational or
stably rational varieties admit
 a
Chow-theoretic (and a fortiori cohomological) decomposition of the
diagonal. More generally, if $X$ is a unirational variety admitting
a unirational parametrization $\mathbb{P}^n\dashrightarrow X$ of
degree $N$, then there is a decomposition
$$N\Delta_X=N (X\times x)+Z\,\,{\rm in}\,\,{\rm CH}^n(X\times X),$$
with $Z$ supported on $D\times X$ for some $D\subsetneqq X$.

On the other hand, the existence of a decomposition of the diagonal
is certainly not a sufficient condition for stable rationality, as
there are surfaces of general type (hence very far from being
rational or stably rational) which admit a Chow-theoretic
decomposition of the diagonal (see Corollary \ref{coro26juill}). It
could be also the case that a smooth projective variety $X$ admits
unirational parametrizations of coprime degrees $N_i$ without being
stably rational (although we do not know such examples).
Nevertheless, the existence of a decomposition of the diagonal is a
rather strong condition, and there are now a number of unirational
examples where the non-existence provides an obstruction to
rationality or stable rationality:

1) Examples of rationally connected  varieties with no cohomological
decomposition of the diagonal include varieties with non-trivial
Artin-Mumford invariant (this is the torsion in $H^3(X,\mathbb{Z})$
or the second unramified cohomology group with torsion
coefficients), or  non-trivial third unramified cohomology group
$H^3_{nr}(X,\mathbb{Q}/\mathbb{Z})$
 (see  \cite{ctoj}  for examples), as  both of these groups have
 to be $0$ when $X$ has a cohomological decomposition of
the diagonal, see \cite{voisinjag}).

2) Examples of rationally connected varieties with no Chow-theoretic
decomposition of the diagonal include varieties with non-trivial
unramified cohomology groups $H^i_{nr}(X,\mathbb{Q}/\mathbb{Z})$
with $i\geq4$, as these groups have
 to be $0$ when $X$ has a Chow-theoretic decomposition of
the diagonal, see \cite{ctv}. We refer to \cite{peyre} for such
examples and to \cite{voisindeg4unra} for the cycle-theoretic
interpretation of this group in degree $4$.

 3) Furthermore,  we proved in \cite{voisindoublesolid}  that
the non-existence of a decomposition of the diagonal is  a criterion
for (stable) irrationality which is actually stronger than those
given by the nontriviality of unramified cohomology: For example, we
show in {\it loc. cit.} that very general smooth quartic double
solids do not admit a Chow-theoretic or even a cohomological
decomposition of the diagonal while their unramified cohomology
vanishes in all positive degrees. We prove similar results for  very
general nodal quartic double solids with $k\leq 7$ nodes and in the
case of very general double solids with  exactly $7$ nodes, we prove
in \cite{voisindoublesolid} that the non-existence of a
cohomological decomposition of the diagonal is equivalent to the
non-existence of a universal codimension $2$ cycle $Z\in {\rm
CH}^2(J(X)\times X)$, where $J(X)$ is the intermediate Jacobian of
$X$. ($J(X)$  is also known to be isomorphic to the group ${\rm
CH}^2(X)_{hom}$ of codimension $2$ cycles homologous to $0$ on $X$.)
Thus, in this case, the study of the decomposition of the diagonal
led to the discovery of new stable birational invariants which are
nontrivial for some unirational varieties.

The purpose of this paper is to investigate the existence of a
decomposition of the diagonal  for cubic hypersurfaces. Motivations
for this study are the following problems:

 1) It is well-known that $3$-dimensional
smooth cubics are irrational (see \cite{clemensgriffiths}), but they
are not known not to be  stably rational.

2) In dimension $4$, some cubics are known to be rational and it is
a famous open problem to prove that there exist irrational cubic
fourfolds. Some precise conjectures concerning the rationality of
cubic fourfolds have been formulated and compared (see
\cite{hassett}, \cite{kuznetsov}, \cite{thoadd}, \cite{addi}). All
these conjectures develop the idea that if a cubic fourfold is
rational, it is related in some way (Hodge-theoretic, categorical)
to a $K3$ surface. An interesting computation has been made recently
by Galkin and Shinder \cite{galkin}, who prove that  a rational
cubic fourfold has to satisfy the property that its variety of lines
is birational to ${\rm Hilb}^2(S)$ for some $K3$ surface $S$, unless
a certain explicitly constructed nonzero element in the Grothendieck
ring of complex varieties is annihilated by the class of
$\mathbb{A}^1$, that is, provides a counterexample to the
cancellation conjecture for the Grothendieck ring. (Note that such
counterexamples are now known to exist by the work of Borisov
\cite{borisov}.)

In general,  the existence of a cohomological decomposition is  much
weaker than the existence of a Chow-theoretic one. Our first result,
which is unconditional for cubic fourfolds and for odd-dimensional
cubics, is the following :
\begin{theo}\label{theoeqchcohdec} Let $X$ be a smooth cubic
hypersurface. Assume that $H^*(X,\mathbb{Z})/H^*(X,\mathbb{Z}
)_{alg}$ has no $2$-torsion (this holds for example if ${\rm
dim}\,X$  is odd  or $ {\rm dim}\,X\leq 4$, or $X$ is very general
of any dimension).
 Then $X$ admits a Chow-theoretic decomposition
of the diagonal (equivalently, ${\rm CH}_0(X)$ is universally
trivial) if and only if it admits a cohomological decomposition of
the diagonal.
\end{theo}
Here $H^*(X,\mathbb{Z} )_{alg}\subset H^*(X,\mathbb{Z})$ is the subgroup of classes of algebraic cycles.
 For any odd degree and odd dimension smooth hypersurface in projective space, the
 quotient group
$H^*(X,\mathbb{Z})/H^*(X,\mathbb{Z} )_{alg}$ has no
$2$-torsion. For cubic hypersurfaces, the first example where we do not know if the assumption is satisfied is $6$-dimensional cubics and degree $6$ integral cohomology classes on them.

 The proof of Theorem \ref{theoeqchcohdec} uses
the fact that ${\rm Hilb}^2(X)$ is birationally a projective bundle
over $X$, a property which is also
 crucially used in the recent  paper
\cite{galkin}.

One consequence of this result, established in Section \ref{sec5},
concerns the following notion:
\begin{Definition}\label{defiuniv} (i) If $Y\subset X$ is a closed algebraic subset of a variety defined over a field $K$,
we say that ${\rm CH}_0(Y)\rightarrow {\rm CH}_0(X) $ is universally
surjective if ${\rm CH}_0(Y_L)\rightarrow {\rm CH}_0(X_L) $ is
surjective for any field $L$ containing $K$.

(ii) The essential ${\rm CH}_0$-dimension of a variety $X$ is the
minimal integer $ k$ such that there exists a closed algebraic
subset $Y\subset X$ of dimension $k$, such that ${\rm
CH}_0(Y)\rightarrow {\rm CH}_0(X)$ is universally surjective.
\end{Definition}
\begin{theo}\label{theonewintro}
The essential ${\rm CH}_0$-dimension of a very general
$n$-dimensional cubic  hypersurface over $\mathbb{C}$
 is either $n$ or $0$.
\end{theo} More precisely,  Theorem \ref{theochowinterdim}
proves the result above for a smooth cubic hypersurface of dimension
$n$ with no $2$-torsion in $H^{n}(X,\mathbb{Z})/
H^n(X,\mathbb{Z})_{alg}$ and such that ${\rm
End}_{HS}H^n(X,\mathbb{Q})_{prim}=\mathbb{Q}$. In dimension $4$, we
get further precise consequences, for example we prove that a cubic
fourfold which is special in the sense of Hassett \cite{hassett},
with discriminant not divisible by $4$, has universally  trivial
${\rm CH}_0$ group.

 The rest of the paper
focuses on the cohomological decomposition of the diagonal. We first
investigate the existence of a cohomological decomposition of the
diagonal for varieties whose non-algebraic cohomology is supported
in middle degree, like complete intersections. We prove the
following result:
\begin{theo} \label{theocohdecomp} Let $X$ be a smooth projective variety such that
$H^*(X,\mathbb{Z})$ has no torsion. Assume that
 $H^{2i}(X,\mathbb{Z})$ is generated over $\mathbb{Z}$ by algebraic cycles
 for $2i\not=n={\rm dim}\,X$.
 Then $X$ admits
a cohomological decomposition of the diagonal if and only if there
exist varieties $Z_i$ of dimension $n-2$,  correspondences
$\Gamma_i\in {\rm CH}^{n-1}(Z_i\times X)$, and integers $n_i$ with
the property that for $\alpha,\,\beta\in H^n(X,\mathbb{Z})$,
\begin{eqnarray}
\label{eqcond} \sum_in_i\langle
\Gamma_i^*\alpha,\Gamma_i^*\beta\rangle_{Z_i}=\langle
\alpha,\beta\rangle_X.
\end{eqnarray}
\end{theo}

Note that the condition (\ref{eqcond}) presents obvious similarities
with the one considered  in \cite{shen} by Shen, who studied the
case of cubic fourfolds. It is however weaker in several respects:
the integers $n_i$ do not need to be positive, and the
correspondences $\Gamma_i$ do not need to factor through the variety
of lines. Finally, the condition formulated by Shen is only
conjecturally a necessary condition for rationality, while our
condition is actually a necessary condition for the triviality of
the universal ${\rm CH}_0$ group, hence a fortiori for (stable)
rationality.
\begin{rema}{\rm Concerning the second assumption in Theorem \ref{theocohdecomp}, it is satisfied by
uniruled threefolds by \cite{voisinuniruled}, but it is not clear that it is satisfied by
Fano complete intersections in any dimension. The group
$H^{2i}(X,\mathbb{Z}),\,2i\not=n$, is equal to $\mathbb{Z}$ by
Lefschetz hyperplane restriction theorem, but Koll\'ar \cite{kollar}
exhibits examples of hypersurfaces where this group is not generated
by an algebraic class for $2i>n$. It is not known if such Fano
examples can be constructed.}
\end{rema}

The case of rationally connected threefolds $X$ is also particularly
interesting. In this case, we complete the results of
\cite{voisinjag} by proving the following result. Let $J(X)$ be the
intermediate Jacobian of $X$. It is isomorphic as a group to ${\rm
CH}^2(X)_{hom}$ via the Abel-Jacobi map. It is canonically a
principally polarized abelian variety, the polarization being
determined by the intersection pairing on $H^3(X,\mathbb{Z})/{\rm
torsion}\cong H^1(J(X),\mathbb{Z})$. Let $\theta\in
H^2(J(X),\mathbb{Z})$ be the class of the Theta divisor of $J(X)$.
\begin{theo} \label{theocohdecomp3fold} (See also Theorem
\ref{cororc3}) Let $X$ be a rationally connected threefold. Then $X$
admits a cohomological decomposition of the diagonal if and only if
the following three conditions are satisfied:
\begin{enumerate}
\item \label{item1} $H^3(X,\mathbb{Z})$ has no torsion.
\item \label{item2} There exists a universal codimension $2$ cycle in $X\times J(X)$.
\item \label{item3} The minimal class $\theta^{g-1}/(g-1)!$ on $J(X)$, ${\rm dim}\,J(X)=g$, is
algebraic, that is, the class of a $1$-cycle in $J(X)$.
\end{enumerate}
\end{theo}
The main new result in this theorem is the fact that condition
\ref{item3} above is implied by the existence of a cohomological
decomposition of the diagonal. In particular, it is a necessary
condition for stable rationality. This can be seen as a variant of
Clemens-Griffiths criterion which can be stated as saying  that a
necessary criterion for rationality of a threefold is the fact that
the  minimal class $\theta^{g-1}/(g-1)!$ on $J(X)$ is the class of
an {\it effective} curve in $J(X)$.

 Note that examples of unirational threefolds not
satisfying \ref{item1} were constructed by Artin and Mumford
\cite{AM}, and examples of unirational threefolds not satisfying
\ref{item2} were constructed in \cite{voisindoublesolid}. It is not
known if examples not satisfying \ref{item3} exist. More generally,
it is not known if there exists any principally polarized abelian
variety $(A,\Theta)$ such that the minimal class
$\theta^{g-1}/(g-1)!$ is not algebraic on $A$, where $g={\rm
dim}\,A$. Notice that  for many Fano threefolds, the intermediate
Jacobian $J(X)$ is a Prym variety, so the class
$2\theta^{g-1}(g-1)!$ is known to be algebraic. In the case of cubic
threefolds, the algebraicity of $\theta^4/4!$ is a classical
completely open problem. Combining the theorems above, we get in
this case:
\begin{theo}\label{theocubthree} Let $X$ be a smooth cubic threefold.
Then $X$ has universally trivial ${\rm CH}_0$ group if and only if
the class $\theta^4/4!$ on $J(X)$ is algebraic. This happens (at
least) on a  countable union of closed subvarieties of codimension
$\leq3$
 of the moduli space of $X$.
\end{theo}

The paper is organized as follows: Theorem \ref{theoeqchcohdec} is
proved in Section \ref{sec2}. Theorem \ref{theocohdecomp} is proved
in Section \ref{sec3} and Theorem \ref{theocohdecomp3fold} is proved
in Section \ref{sec4}. In Section \ref{sec5}, we come back to the
case of cubic hypersurfaces, where we prove Theorem
\ref{theochowinterdim} and establish further results, particularly
in dimension  $4$.

\vspace{0.5cm}

 {\bf Thanks.} {\it I
thank Arnaud Beauville, Jean-Louis Colliot-Th\'{e}l\`{e}ne and Brendan
Hassett for interesting discussions related to this paper. I also
thank Burt Totaro for  indicating the  references
\cite{acta}, \cite{milgram}, for explaining  me the results
proved there, and also for writing the very useful paper \cite{totaro}. Finally, I am very  grateful to the referee for
 his/her careful reading and criticism.}
\section{Chow-theoretic and cohomological decomposition of the diagonal
\label{sec2}} We prove in this section Theorem \ref{theoeqchcohdec}.
In the case of cubic hypersurfaces of dimension $\leq4$, a shorter
proof will be given, which uses   the following result of
independent interest:
  \begin{prop}\label{propredalg} Let $X$ be a smooth projective variety.
  If $X$ admits a decomposition of the diagonal modulo algebraic equivalence,
  that is
  \begin{eqnarray}
  \label{eqpouredalg} \Delta_X-X\times x=Z\,\,{\rm in}\,\,{\rm CH}(X\times X)/{\rm alg},
  \end{eqnarray}
  with $Z$ supported on $D\times X$ for some closed algebraic subset
  $D\subsetneqq X$,
  then $X$ admits a Chow-theoretic decomposition of the diagonal.
  \label{proredalg}
  \end{prop}
  \begin{proof} We use the fact proved in \cite{voe}, \cite{voivoe} that cycles
  algebraically equivalent to $0$ are nilpotent for the composition of self-correspondences.
  We write (\ref{eqpouredalg}) as
  \begin{eqnarray}
  \label{eqpouredalg2} \Delta_X-X\times x-Z=0\,\,{\rm in}\,\,{\rm CH}(X\times X)/{\rm alg}
  \end{eqnarray}
  and apply the nilpotence result mentioned above. This provides
  \begin{eqnarray}
  \label{eqpouredalg3} (\Delta_X-X\times x-Z)^{\circ N}=0\,\,{\rm in}\,\,{\rm CH}(X\times X)
  \end{eqnarray}
  for some large $N$.
  As $\Delta_X-X\times x$ is a projector and $Z\circ (X\times x)=0$, this gives
  \begin{eqnarray}
  \label{eqpouredalg4} \Delta_X-X\times x-W\circ Z=0\,\,{\rm in}\,\,{\rm CH}(X\times X),
  \end{eqnarray}
  for some cycle $W$ on $X\times X$. As $W\circ Z$ is supported on $D\times X$, for some
  $D\subsetneqq X$, this concludes the proof.
  \end{proof}
  \begin{coro} \label{coro26juill} Let $S$ be a surface of general type with ${\rm CH}_0(S)=\mathbb{Z}$ and
${\rm Tors}(H^*(S,\mathbb{Z}))=0$ (for example the Barlow surface
(\cite{barlow}). Then $S$ has universally trivial ${\rm CH}_0$
group.
\end{coro}
\begin{proof} (See \cite{ACTP} for a different proof.) As $p_g(S)=q(S)=0$ by Mumford's theorem \cite{mumford},
the cohomology $H^*(S,\mathbb{Z})$  is generated by classes of
algebraic cycles. As $H^*(S,\mathbb{Z})$ has no torsion, the
cohomology of $S\times S$ admits a K\"unneth decomposition with
integral coefficients, so that we can write the class of the
diagonal of $S$ as
\begin{eqnarray}
\label{eqintrodiag}[\Delta_S]=\sum_i[\alpha_i]\otimes
[\beta_i]\,\,{\rm in}\,\,H^4(S\times S,\mathbb{Z}),
\end{eqnarray}
where $\alpha_i$, $\beta_i$ are algebraic cycles on $S$ with ${\rm
dim}\,\alpha_i+{\rm dim}\,\beta_i=2$. Clearly, $[\alpha_i]\otimes
[\beta_i]=[\alpha_i\times \beta_i]$ is supported over $D\times S$
with $D\subsetneqq S$ when ${\rm dim}\,\alpha_i<  2$, so that
(\ref{eqintrodiag}) provides in fact a cohomological decomposition
of $\Delta_S$:
\begin{eqnarray}
\label{eqintrodiag2}[\Delta_S]=[S\times s]+[Z]\,\,{\rm
in}\,\,H^4(S\times S,\mathbb{Z}),
\end{eqnarray}
where $Z$ is a cycle supported over $D\times S$, for some
$D\subsetneqq S$.
 Next, as ${\rm CH}_0(S\times S)=\mathbb{Z}$,
codimension $2$ cycles on $S\times S$ which are cohomologous to $0$
are algebraically equivalent to $0$ by \cite{blochsrinivas}. Thus
the cycle $\Gamma:=\Delta_S-S\times s-Z$ is algebraically equivalent
to $0$ on $S\times S$. The surface $S$ thus admits a decomposition
of the diagonal modulo algebraic equivalence, and we then apply
Proposition \ref{propredalg}.
\end{proof}
For a smooth projective variety $X$, we denote by $X^{[2]}$ the
second punctual Hilbert scheme of $X$. It is smooth, obtained as the
quotient of the blow-up $\widetilde{X\times X}$ of $X\times X$ along
the diagonal by its natural involution. Let $\mu:X\times
X\dashrightarrow X^{[2]}$ be the natural rational map and
$r:\widetilde{X\times X}\rightarrow X^{[2]}$ be the quotient
morphism.
 We start with the following result:
\begin{lemm} \label{le1}Let $X$ be a smooth projective variety of dimension $n$. Then there exists a
codimension $n$ cycle $Z$ in $X^{[2]}$ such that $\mu^*Z=\Delta_X$
in ${\rm CH}^n(X\times X)$.
\end{lemm}
\begin{proof} Let $E_\Delta$ be the exceptional divisor over the diagonal
of the blow-up map  $\tau: \widetilde{X\times X}\rightarrow X\times
X$. The key point is the fact that there is a (non effective)
divisor $\delta$ on $X^{[2]}$ such that $r^*\delta=E_\Delta$. It
follows that
$$r^*\delta^n=E_\Delta^n\,\,{\rm in}\,\,{\rm CH}^n(\widetilde{X\times X}).$$
Now we use the fact that
$$\tau_*((-1)^{n-1}E_\Delta^n)=\Delta_X \,\,{\rm in}\,\,{\rm CH}^n({X\times X}),$$
which gives
$\mu^*((-1)^{n-1}\delta^n)=\tau_*(r^*((-1)^{n-1}\delta^n))=\Delta_X
 \,\,{\rm in}\,\,{\rm CH}^n({X\times X})$.
\end{proof}
\begin{coro} \label{corosympull} Any symmetric codimension $n$ cycle on $X\times X$ is rationally equivalent to
$\mu^*\Gamma$ for a codimension $n$ cycle $\Gamma$ on $X^{[2]}$.
\end{coro}
\begin{proof} Indeed, we can write $Z=Z_1+Z_2$ where $Z_1$ is a combination of
irreducible subvarieties of $X\times X$ invariant under the
involution $i$ of $X\times X$, and $Z_2$ is of the form
$Z'_2+i(Z'_2)$, where the diagonal does not appear in $Z'_2$. Write
$Z_1=n_1\Delta_X+Z'_1$, with $Z'_1=\sum_jn_j Z'_{1,j}$, the $Z'_j$
being invariant under $i$ but different from $\Delta_X$. Then
$Z'_{1,j}$ is the inverse image of a subvariety $Z''_{1,j}$ of
$X^{(2)}$. Let $\widetilde{Z''_{1,j}}$ be the proper transform of
 $Z''_{1,j}$ under the Hilbert-Chow map
$X^{[2]}\rightarrow X^{(2)}$. Then clearly
$\mu^*(\widetilde{Z''_{1,j}})=Z'_{1,j}$ in ${\rm CH}^n(X\times X)$
so
$$\mu^*(\sum_j n_j\widetilde{Z''_{1,j}})=Z'_1\,\,{\rm in}\,\,{\rm CH}^n(X\times X).$$
Next, let $\overline{Z'_2}$ be the image of $Z'_2$ in $X^{[2]}$ by
$\mu$. Then clearly
$$\mu^*(\overline{Z'_2})=Z'_2+i(Z'_2)=Z_2\,\,{\rm in}\,\,{\rm CH}^n(X\times X).$$
Thus
$$Z=n_1\Delta_X+Z'_1+Z_2=n_1\Delta_X+\mu^*(\sum_jn_j\widetilde{Z''_{1,j}})+\mu^*(\overline{Z'_2}).$$
 Finally, we use Lemma \ref{le1} to
conclude.
\end{proof}
We have next:
\begin{lemm} \label{lesym}Suppose $X$ admits a cohomological decomposition of the diagonal
\begin{eqnarray}
\label{eqdecomp8juin} [\Delta_X-x\times X]=[Z]\,\,{\rm in}\,\,H^{2n}(X\times X,\mathbb{Z}),
\end{eqnarray}
where $Z$ is a cycle supported on $D\times X$ for some proper closed
algebraic subset $D$ of $X$. Then $X$ admits a cohomological
decomposition of the diagonal
\begin{eqnarray}
\label{eqdecomp8juin2} [\Delta_X-x\times X-X\times x]=[W]\,\,{\rm in}\,\,H^{2n}(X\times X,\mathbb{Z}),
\end{eqnarray}
where $W$ is  a cycle supported on $D\times X$ for some
$D\subsetneqq X$, and $W$ is invariant under the involution $i$.
\end{lemm}
\begin{proof} Let us denote by $^t\Gamma$ the image of a cycle $\Gamma$ under the involution $i$ of $X\times X$.
Formula (\ref{eqdecomp8juin}) gives as well
$$ [\Delta_X-X\times x]=[^tZ]\,\,{\rm in}\,\,H^{2n}(X\times X,\mathbb{Z}),$$
and
\begin{eqnarray}
\label{eqdecomp8juin3}[(\Delta_X-X\times x)\circ (\Delta_X-x\times X)]=[^tZ\circ Z]\,\,{\rm in}\,\,H^{2n}(X\times X,\mathbb{Z}).
\end{eqnarray}
The cycle $^tZ\circ Z$ is invariant under the involution and
supported on $D\times X$. On the other hand, the left hand side in
(\ref{eqdecomp8juin3}) is equal to $[\Delta_X-X\times x-x\times X]$
(we assume here $n>0$).
\end{proof}
The following proposition is a key point in our proof of Theorem
\ref{theoeqchcohdec}.
\begin{prop} \label{leGamma} Let
$X$ be a smooth odd degree  complete intersection  in projective
space. If $X$ admits a cohomological decomposition of the diagonal,
and $H^{2*}(X,\mathbb{Z} )/H^{2*}(X,\mathbb{Z})_{alg}$ has no $2$-torsion, there
exists a cycle $\Gamma\in {\rm CH}^n(X^{[2]})$ with the following
properties:

(i) $\mu^*\Gamma= \Delta_X-x\times X-X\times x-W$ in ${\rm
CH}^n(X\times X)$, with $W$ supported over $D\times X$, for some
closed algebraic subset $D\subsetneqq X$.

(ii) $[\Gamma]=0$  in $H^{2n}(X^{[2]},\mathbb{Z})$.
\end{prop}
The proof of this proposition will use a few lemmas concerning the
cohomology of $X^{[2]}$, when $X$ is the projective space or a
smooth complete intersection in projective space. We give here an elementary and purely algebraic proof. These results
can be obtained as well as an application of \cite{acta} or
\cite{milgram} (cf. \cite{totaro}), but these papers are written in a topologist's
language and it is not obvious how to translate them into the concrete
statements below. With a better understanding of these papers, our
 arguments would presumably prove Proposition \ref{leGamma} for a smooth projective variety
 $X$ such that $H^*(X,\mathbb{Z})$ is torsion free and $H^{2*}(X,\mathbb{Z} )/H^{2*}(X,\mathbb{Z})_{alg}$ has no $2$-torsion.

 Let $X$ be a smooth projective variety and let
$$j_{E,X}:E_{\Delta,X}\hookrightarrow X^{[2]},\,\,i_{E,X}:E_{\Delta,X}\rightarrow \widetilde{X\times X},\,\,\tau_{E,X}:E_{\Delta,X}\rightarrow X,$$
be respectively the inclusion of the exceptional divisor over the
diagonal in $X^{[2]}$, its inclusion in $\widetilde{X\times X}$, and its
natural morphism to $X$. Let $\delta\in {\rm Pic}\,X^{[2]}$ be the
natural divisor such that $2\delta=E$. Then $\delta_E:=\delta_{\mid
E_{\Delta,X}}$ is the line bundle $\mathcal{O}_{E_{\Delta,X}}(-1)$ of the projective bundle
$\tau_{E,X}:E_{\Delta,X}\cong \mathbb{P}(T_X)\rightarrow X$.
\begin{lemm} \label{lecox2projspace} Assume $X=\mathbb{P}^n$ and let $h=c_1(\mathcal{O}_{\mathbb{P}^n}(1))$.
 For any   cohomology class $a\in H^*(E_{\Delta,\mathbb{P}^n},\mathbb{Z})$, one has
 $(j_{E,\mathbb{P}^n})_*a\in 2H^{*+2}((\mathbb{P}^n)^{[2]},\mathbb{Z})$
if and only if \begin{eqnarray}
\label{eq13}
a=\sum_{i\geq 0, m\geq
0}\alpha_{i,m}\delta_E^i\cdot \tau_{E,\mathbb{P}^n}^*h^m\cdot (\tau_{E,\mathbb{P}^n}^*h-\delta_E)^m\,\,{\rm
mod.} \,\,2H^*(E_{\Delta,\mathbb{P}^n},\mathbb{Z}),
\end{eqnarray}
where the $\alpha_{i,m}$ are integers.
\end{lemm}
\begin{proof} To see that the condition is sufficient (and also to get a nice interpretation
of the condition), we observe that for any smooth subvariety
$\Sigma\subset \mathbb{P}^n$ of codimension $m$, we have the inclusion
$\Sigma^{[2]}\subset (\mathbb{P}^n)^{[2]}$ and  the class
$b:=[\delta\cdot \Sigma^{[2]}]\in H^{4m+2}((\mathbb{P}^n)^{[2]},\mathbb{Z})$ satisfies
$2b=[E_{\Delta,\mathbb{P}^n}
\cdot \Sigma^{[2]}]$, so that
$$2b=(j_{E,\mathbb{P}^n})_*([ \Sigma^{[2]}]_{\mid E_{\Delta,\mathbb{P}^n}})\,\,{\rm in}\,\,
H^{4m+2}((\mathbb{P}^n)^{[2]},\mathbb{Z}).$$
Next we observe that $\Sigma^{[2]}\cap
E_{\Delta,\mathbb{P}^n}$ is equal to $\mathbb{P}(T_\Sigma)\subset \mathbb{P}(T_{\mathbb{P}^n})$. Take
now for $\Sigma$ a $\mathbb{P}^{n-m}\subset \mathbb{P}^n$. Then
the class of $ \mathbb{P}(T_\Sigma)\subset  \mathbb{P}(T_{\mathbb{P}^n})$ is
$\tau_{E,\mathbb{P}^n}^*h^m\cdot (\tau_{E,\mathbb{P}^n}^*h-\delta_E)^m$. We thus proved that
$(j_{E,\mathbb{P}^n})_*(\tau_{E,\mathbb{P}^n}^*h^m\cdot (\tau_{E,\mathbb{P}^n}^*h-\delta_E)^m)$ is divisible by $2$ in
$H^*((\mathbb{P}^n)^{[2]},\mathbb{Z})$, and so is the class $$
(j_{E,\mathbb{P}^n})_*(\delta_E^i\cdot\tau_{E,\mathbb{P}^n}^*h^m\cdot (\tau_{E,\mathbb{P}^n}^*h-\delta_E)^m)=\delta^i\cdot(j_{E,\mathbb{P}^n})_*(\tau_{E,\mathbb{P}^n}^*h^m\cdot (\tau_{E,\mathbb{P}^n}^*h-\delta_E)^m)
$$  for any $i\geq 0$.

In the other direction, it is better to see $(\mathbb{P}^n)^{[2]}$ as a
$\mathbb{P}^{2}$-bundle over the Grassmannian $G(2,n+1)$, namely, if
$\pi_1: P\rightarrow G(2,n+1)$ is the universal
$\mathbb{P}^{1}$-bundle, it is clear that
$(\mathbb{P}^n)^{[2]}$ is isomorphic to the second symmetric
product $\pi_2: P_2\rightarrow G(2,n+1)$ of $P$ over $G(2,n+1)$.
Furthermore, $E_{\Delta,\mathbb{P}^n}\subset (\mathbb{P}^n)^{[2]}$ identifies with the Veronese embedding
$P\subset P_2$. Write $P=\mathbb{P}(\mathcal{E})$ with polarization
$H=\tau_{E,\mathbb{P}^n}^*h$; then $P_2=\mathbb{P}(S^2\mathcal{E})$ with
polarization $H_2$, and clearly \begin{eqnarray}
\label{eqdu9sep1}
j_{E,\mathbb{P}^n}^*H_2=2H\,\,{\rm in}\,\,H^2(P,\mathbb{Z})
\end{eqnarray}
since $j_{E,\mathbb{P}^n}:P\rightarrow P_2$ is the Veronese embedding.
The cohomology of $P$ decomposes as
$$H^*(P,\mathbb{Z})=\pi_1^*H^*(G(2,n+1),\mathbb{Z})\oplus H\cdot
\pi_1^*H^{*-2}(G(2,n+1),\mathbb{Z}).$$
We claim that modulo $2$, the set of classes $a$ as in (\ref{eq13})
is exactly $\pi_1^*H^*(G(2,n+1),\mathbb{Z}/2\mathbb{Z})$.
Let
$l=c_1(\mathcal{E}),\,c=c_2(\mathcal{E})$ be the two generators of
$H^*(G(2,n+1),\mathbb{Z})$. It is easy to check that
$\delta=H_2-\pi_2^*l$ in $H^2(P_2,\mathbb{Z})$. Restricting this equality to $P$,  we get by (\ref{eqdu9sep1})
\begin{eqnarray}
\label{eqdu9sep2}\delta_E=\pi_1^*l\,\,{\rm mod.}\,\,
2H^2(P,\mathbb{Z}).
\end{eqnarray}
 Finally we have
$\pi_1^*c=H\cdot (\pi_1^*l-H)$ in $H^*(P,\mathbb{Z})$,
hence we get by (\ref{eqdu9sep2})
\begin{eqnarray}
\label{eqdu9sep3}\pi_1^*c=
\tau_{E,\mathbb{P}^n}^*h\cdot(\tau_{E,\mathbb{P}^n}^*h-\delta_E) \,\,{\rm mod.}\,\,2H^*(P,\mathbb{Z}),
\end{eqnarray}
which together with (\ref{eqdu9sep2}) proves the claim.
 Having this, the fact that condition (\ref{eq13}) is  sufficient tells us that
 $$(j_{E,\mathbb{P}^n})_*\circ\pi_1^*(H^*(G(2,n+1),\mathbb{Z}))\subset
 2H^*(P_2,\mathbb{Z})$$
 and  in order  to prove that it is necessary, we need to prove that the set of classes $z\in H^*(P,\mathbb{Z})$
 such that $(j_{E,\mathbb{P}^n})_*z\in 2H^*(P_2,\mathbb{Z})$ is equal to $\pi_1^*(H^*(G(2,n+1),\mathbb{Z}))$
 modulo $2H^*(P,\mathbb{Z})$.
 Equivalently, we have to prove that
 if
 $z\in  H^*(G(2,n+1),\mathbb{Z})$ satisfies
$(j_{E,\mathbb{P}^n})_*(H\cdot \pi_1^*z)\in 2H^*(P_2,\mathbb{Z})$,
then $z\in 2H^*(G(2,n+1),\mathbb{Z})$. But for any $\alpha\in
H^*(G(2,n+1),\mathbb{Z})$, we have
$$\langle \alpha,z\rangle_{G(2,n+1)}=\langle \pi_1^*\alpha,H\cdot \pi_1^*z\rangle_{P}=
\langle \pi_2^*\alpha,(j_{E,\mathbb{P}^n})_*(H\cdot
\pi_1^*z)\rangle_{P_2},$$ which implies that  $\langle
\alpha,z\rangle_{G(2,n+1)}$ is even since
$(j_{E,\mathbb{P}^n})_*(H\cdot \pi_1^*z)$ is divisible by $2$. Hence
$z$ is divisible by $2$ by Poincar\'e duality on $G(2,n+1)$.
\end{proof}
\begin{lemm} \label{lecox2CI} Let $X\subset \mathbb{P}^N$ be a smooth odd
degree complete intersection of dimension $n$.

(i) For any integer $m\geq0$, the class
$(j_{E,X})_*(\tau_{E,X}^*h^m\cdot (\tau_{E,X}^*h-\delta_E)^m)\in
H^{4m+2}(X^{[2]},\mathbb{Z})$ is equal to $2
\delta\cdot[\Sigma_l^{[2]}]$, where $\Sigma_m\subset X$ is the
smooth proper intersection of $X$ with a linear space of codimension
$m$ in $\mathbb{P}^N$.

(ii)  For any  integral cohomology class $a\in H^{2n-2}(E_{\Delta,X},\mathbb{Z})$,
one has
 $(j_{E,X})_*a\in 2H^{2n}(X^{[2]},\mathbb{Z})$
if and only if \begin{eqnarray} \label{eqvanpourci} a=\sum_{i\geq 0,
m\geq 0, i+2m=n-1}\alpha_{i,m}\delta_E^i \cdot \tau_{E,X}^*h^m\cdot
(\delta_E-\tau_{E,X}^*h)^m\,\,{\rm modulo} \,\,2H^*(E_{\Delta,X},\mathbb{Z}),
\end{eqnarray}
where the $\alpha_{i,m}$'s are integers.

\end{lemm}
\begin{proof} (i) This has been already proved in the case of $\mathbb{P}^N$ and follows from the
fact that $\tau_{E,X}^*h^m\cdot
(\tau_{E,X}^*h-\delta_E)^m\in H^{4m}(E_{\Delta,X},\mathbb{Z})$ is the class of $\mathbb{P}(T_{\Sigma_m})=E_{\Delta,\Sigma_m}$
in $\mathbb{P}(T_X)=E_{\Delta,X}$.

(ii) In the case where $X\stackrel{j_X}{\hookrightarrow
}\mathbb{P}^N$ has odd dimension, observe that the map
$j_{X*}:H^*(X,\mathbb{Z}/2)\rightarrow
H^{*+2k}(\mathbb{P}^N,\mathbb{Z}/2) $, $k:=N-n$, is injective since
$X$ has odd degree. It follows as well that denoting by
$\tilde{\jmath}_X:\mathbb{P}(T_X)\rightarrow
\mathbb{P}(T_{\mathbb{P}^N})$ the natural map,
$\tilde{\jmath}_{X*}:H^*(\mathbb{P}(T_X),\mathbb{Z}/2)\rightarrow
H^{*+4k}(\mathbb{P}(T_{\mathbb{P}^N}),\mathbb{Z}/2)$ is also
injective. Now, let $$a\in
H^{2n-2}(E_{\Delta,X},\mathbb{Z})=H^{2n-2}(\mathbb{P}(T_X),\mathbb{Z})$$
such that $(j_{E,X})_*a\in 2H^{2n}(X^{[2]},\mathbb{Z})$. Then
$\tilde{\jmath}_{X*}a\in H^*(E_{\Delta,\mathbb{P}^N},\mathbb{Z})$
satisfies $(j_{E,\mathbb{P}^N})_*(\tilde{\jmath}_{X*}a)\in
2H^{*}((\mathbb{P}^N)^{[2]},\mathbb{Z})$. Thus we conclude by Lemma
\ref{lecox2projspace} that the class $\tilde{\jmath}_{X*}a$ mod. $2$
belongs to the subgroup of
$H^*(E_{\Delta,\mathbb{P}^N},\mathbb{Z}/2)$ generated by the
$\delta_E^i \cdot\tau_{E,\mathbb{P}^n}^*h^m\cdot
(\delta_E-\tau_{E,\mathbb{P}^n}^*h)^m$ with $i\geq 0,\,m\geq 0$. It
easily  follows that $a$ mod. $2$ belongs to the subgroup of
$H^*(E_{\Delta,X},\mathbb{Z}/2)$ generated by the $\delta_E^i
\cdot\tau_{E,X}^*h^m\cdot (\delta_E-\tau_{E,X}^*h)^m$ with $i\geq
0,\,l\geq 0$, since the class of $E_{\Delta,X}$ in
$E_{\Delta,\mathbb{P}^N}$ is equal modulo $2$ to
$\tau_{E,\mathbb{P}^n}^*h^{N-n}\cdot
(\delta_E-\tau_{E,\mathbb{P}^n}^*h)^{N-n}$.

In the case where $X$ has even dimension, the maps
$j_{X*}:H^{n}(X,\mathbb{Z}/2)\rightarrow
H^{n+2k}(\mathbb{P}^N,\mathbb{Z}/2)$ and
$\tilde{\jmath}_{X*}:H^{2n-2}(\mathbb{P}(T_X),\mathbb{Z}/2) \rightarrow
H^{2n-2+4k}(\mathbb{P}(T_{\mathbb{P}^n}),\mathbb{Z}/2)$
 are no longer injective, but their kernels are equal respectively
 to $H^n(X,\mathbb{Z}/2)_{prim}$ and $\delta_E^{n/2-1}\cdot \tau_{E,X}^*H^n(X,\mathbb{Z}/2)_{prim}$.
 The proof above shows that if
 $a\in H^{2n-2}(E_{\Delta,X},\mathbb{Z})$ satisfies
 $(j_{E,X})_*a\in 2H^{2n}(X^{[2]},\mathbb{Z})$, then
 $a\in \delta_E^{n/2-1}\cdot \tau_{E,X}^*H^n(X,\mathbb{Z}/2)_{prim}$ modulo
 the subgroup of
$H^*(E_{\Delta,X},\mathbb{Z}/2)$ generated by the
$\delta_E^i \cdot\tau_{E,X}^*h^m\cdot
(\delta_E-\tau_{E,X}^*h)^m$ with $i\geq 0,\,m\geq 0$. By (i), it thus suffices to prove
that a class
$a\in \delta_E^{n/2-1}\cdot \tau_{E,X}^*H^n(X,\mathbb{Z}/2)_{prim}$
satisfying $(j_{E,X})_*a=0$ in $H^{2n}(X^{[2]},\mathbb{Z}/2)$ must be $0$.
 This is proved as follows: By a monodromy argument, either any  class
 $a \in \delta_E^{n/2-1}\cdot \tau_{E,X}^*H^n(X,\mathbb{Z}/2)_{prim}$ satisfies
 this property, or no nonzero class satisfies it.
 Assume that the first possibility occurs. Then we get a contradiction as follows:
 Let $$a=\delta_E^{n/2-1}\cdot \tau_{E,X}^* \alpha,\,b=\delta_E^{n/2-1}\cdot \tau_{E,X}^* \beta
 \in H^{2n-2}(E_{\Delta,X},\mathbb{Z}),$$
 with $\alpha,\,\beta\in H^{n}(X,\mathbb{Z})_{prim}$.
 Then
 \begin{eqnarray}
 \label{eq10mai1}\langle j_{\Delta,X*}a,j_{\Delta,X*}b\rangle_{X^{[2]}}=-2\langle \alpha,\beta\rangle_X.
 \end{eqnarray}
If both $(j_{\Delta,X})_*\alpha$ and $(j_{\Delta,X})_*\beta$ are divisible by $2$ in
$H^{2n}(X^{[2]},\mathbb{Z})$, then $\langle (j_{\Delta,X})_*a,(j_{\Delta,X})_*b\rangle_{X^{[2]}}$
is divisible by $4$, hence $\langle \alpha,\beta\rangle_X$ is divisible by
$2$ by (\ref{eq10mai1}). Thus, under our assumption, the intersection pairing modulo $2$
would be identically $0$ on $H^{n}(X,\mathbb{Z}/2)_{prim}$.
 As $X$ has odd degree, the intersection pairing
is nondegenerate on $H^{n}(X,\mathbb{Z}/2)_{prim}$ and we get a contradiction.
\end{proof}
\begin{proof}[Proof of Proposition \ref{leGamma}] Let $X$ be an odd degree complete
intersection in $\mathbb{P}^N$ which admits a cohomological
decomposition of the diagonal. We know by Lemma \ref{lesym} that
there is a symmetric cycle $W$ supported on $D\times X$,
$D\subsetneqq X$, such that $[\Delta_X-x\times X-X\times x]=[W]$ in
$H^{2n}(X\times X)$. Corollary \ref{corosympull} provides
 a cycle $\Gamma_0 \in {\rm CH}^n(X^{[2]})$ such that $\mu^*\Gamma_0= \Delta_X-x\times X-X\times x-W$
 in ${\rm CH}^n(X\times X)$, which is property (i). It remains to see that we
 can modify $\Gamma_0$ keeping property (i) and
 imposing condition (ii), namely
 $$[\Gamma_0]=0\,\,{\rm in}\,\,H^{2n}(X^{[2]},\mathbb{Z}).$$
 We know that $\mu^*[\Gamma_0]=0$ in $H^{2n}(X\times X,\mathbb{Z})$, which implies that
 $r^*[\Gamma_0]$ vanishes in $H^{2n}(\widetilde{X\times X}\setminus E_\Delta,\mathbb{Z})$. Thus there is a
 cohomology class $\beta\in H^{2n-2}(E_\Delta,\mathbb{Z})$ such that
 $$i_{E*}\beta= r^*[\Gamma_0]\,\,{\rm in}\,\,
 H^{2n}(\widetilde{X\times X},\mathbb{Z}),$$
 where we come back to the notation $i_E:E_\Delta\rightarrow \widetilde{X\times X}$,
 $j_E:E_\Delta\rightarrow X^{[2]}$ for the natural inclusions of the exceptional divisor over the diagonal
 of $X$.
 This implies that $j_{E*}\beta$ is divisible by $2$ in
 $H^{2n}(X^{[2]},\mathbb{Z})$. Indeed, we have
 \begin{eqnarray}
 \label{eqpour6mai}j_{E*}\beta=r_*(i_{E*}\beta)=r_*(r^*([\Gamma_0]))=2 [\Gamma_0]\,\,{\rm
 in}\,\,H^{2n}(X^{[2]},\mathbb{Z}).
 \end{eqnarray}
According to Lemma \ref{lecox2CI},(ii),  one has then
\begin{eqnarray}
 \label{eqpour6mai0}
 \beta=\sum_{i\geq
0,i+2m=n-1}\alpha_i\delta_E^i(h-\delta_E)^mh^m+2\gamma\,\,{\rm
in}\,\,H^{2n-2}(E_\Delta,\mathbb{Z}), \end{eqnarray} which by Lemma
\ref{lecox2CI},(i) gives
\begin{eqnarray}
 \label{eqpour6mai2}j_{E*}\beta= 2(\sum_{i\geq
0,i+2m=n-1}\alpha_i\delta^{i+1}[\Sigma_{m}^{[2]}])\,\,\,\,{\rm
 in}\,\,H^{2n}(X^{[2]},\mathbb{Z})
\,\,{\rm modulo}\,\,j_{E*}(2H^*(E_\Delta,\mathbb{Z}))
\end{eqnarray}
for some integers $\alpha_i$. By (\ref{eqpour6mai}), we thus have
\begin{eqnarray}
\label{eqpour6mai3}2(\sum_{i\geq
0,i+2m=n-1}\alpha_i\delta^{i+1}[\Sigma_{m}^{[2]}]+j_{E*}(\gamma))=2
[\Gamma_0]\,\,{\rm
 in}\,\,H^{2n}(X^{[2]},\mathbb{Z})
 \end{eqnarray}
 for some integral homology class $\gamma\in H^{2n-2}(E,\mathbb{Z})$.
 It is proved in \cite[Theorem 2.2]{totaro} that the cohomology of $X^{[2]}$ as no $2$-torsion when $X$ is
 an odd degree complete intersection in projective space, so (\ref{eqpour6mai3}) gives
\begin{eqnarray}
\label{eqpour6mai4} \sum_{i\geq
0,i+2m=n-1}\alpha_i\delta^{i+1}[\Sigma_{m}^{[2]}]+j_{E*}(\gamma)=
[\Gamma_0]\,\,{\rm
 in}\,\,H^{2n}(X^{[2]},\mathbb{Z}).
 \end{eqnarray}
We now observe that $\beta$ is algebraic in
$H^{2n-2}(E_\Delta,\mathbb{Z})$, which easily follows from the fact
that $i_{E*}\beta$ is algebraic in $H^{2n}(\widetilde{X\times X},\mathbb{Z})$ by recalling that $\widetilde{X\times X},\mathbb{Z}$ is simply the blow-up of $X\times X$ along its diagonal.
 We thus conclude from (\ref{eqpour6mai0}) that $2\gamma$ is algebraic,
 and by our assumption that $H^{2*}(X,\mathbb{Z})/H^{2*}(X,\mathbb{Z})_{alg}$ has no $2$-torsion,
 $\gamma$ is algebraic, that is,
 $\gamma=[\Gamma']$ for some  $\Gamma'\in{\rm CH}^{n-1}(E_\Delta)$.
Thus we have
 \begin{eqnarray}\label{eqpour6mai42}\sum_{i\geq
0,i+2m=n-1}\alpha_i\delta^{i+1}[\Sigma_{m}^{[2]}]+j_{E*}[\Gamma']=
[\Gamma_0]\,\,{\rm
 in}\,\,H^{2n}(X^{[2]},\mathbb{Z}).
 \end{eqnarray}
 We have $\mu^*\delta^n=\pm [\Delta_X]$ and, for any $z\in {\rm CH}^{n-1}(E_\Delta)$,
  $\mu^*(j_{E*}z)=N \Delta_X$ for some $N\in \mathbb{Z}$.
 As $\mu^*[\Gamma_0]=0$, we can thus assume,  up to modifying $\Gamma'$ but without changing
 $\mu^*\Gamma_0$, that in formula (\ref{eqpour6mai42}),
$\alpha_{n-1}=0$ and $\Gamma'$ satisfies $\tau_{E*}\Gamma'=0$,
hence  $\mu^*(j_{E*}(\Gamma'))=0$ in ${\rm CH}^n(X\times
X)$. Next, for $m>0$, $\mu^*\delta^{i+1}[\Sigma_{m}^{[2]}]$ is
supported over $D\times X$ for some proper closed algebraic subset
$D\subset X$. It follows that the cycle
$$\Gamma:=\Gamma_0-\sum_{i\geq
0,i+2m=n-1}\alpha_i\delta^{i+1}[\Sigma_{m}^{[2]}]-j_{E*}(\Gamma')$$ is
cohomologous to $0$ on $X^{[2]}$, and satisfies
$\mu^*\Gamma=\mu^*\Gamma_0+Z'$ where $Z'\in {\rm CH}^n(X\times X)$
is supported over $D\times X$.
\end{proof}

We now consider the case where $X$ is a smooth cubic hypersurface in
$\mathbb{P}^{n+1}$. We then have the following description of
$X^{[2]}$, which is also used in \cite{galkin}. We denote below by
$F(X)$ the variety of lines of $X$. Let
$$P=\{([l],x),\,x\in l,\,l\subset X\}$$
be the universal $\mathbb{P}^1$-bundle, with first projection
$p:P\rightarrow F(X)$, and let $P_2\rightarrow F(X)$ be the
$\mathbb{P}^2$-bundle defined as the symmetric product of $P$ over
$F(X)$. There is a natural embedding $P_2\subset X^{[2]}$ which maps
each fiber of $P\rightarrow F(X)$, that is the second symmetric
product of  a line in $X$, isomorphically onto the set of subschemes
of length $2$ of $X$ contained in this line. Let $p_X:P_X\rightarrow
X$ be the projective bundle with fiber over $x\in X$ the set of
lines in $\mathbb{P}^{n+1}$ passing through $x$. Note that $P$ is
naturally contained in $P_X$, as one sees by considering the second
projection $q:P\rightarrow X$.
\begin{prop}\label{propX2} (i) The rational map $\Phi:X^{[2]}\dashrightarrow P_X$ which to a unordered pair of points
$x,\,y\in X$ not contained in a common line of $X$ associates the
pair $([l_{x,y}],z)$, where $l_{x,y}$ is the line in
$\mathbb{P}^{n+1}$ generated by $x$ and $y$ and $z\in X$ is the
residual point  of the intersection $l_{x,y}\cap X$, is
desingularized by the blow-up of $P_2$ in $X^{[2]}$.

(ii) The induced morphism
$\widetilde{\Phi}:\widetilde{X^{[2]}}\rightarrow P_X$ identifies
$\widetilde{X^{[2]}}$ with the blow-up $\widetilde{P_X}$ of $P$ in
$P_X$.

(iii) The  exceptional divisors of the two maps
$\widetilde{X^{[2]}}\rightarrow X^{[2]}$ and
$\widetilde{P_X}\rightarrow P_X$ are identified by the isomorphism
$\widetilde{\Phi}':\widetilde{X^{[2]}}\cong \widetilde{P_X}$ of
(ii).
\end{prop}
\begin{proof} (i) There is a morphism from $X^{[2]}$ to the Grassmannian $G(1,n+1)$ of lines in $\mathbb{P}^{n+1}$,
which to $x+y$ associates the line $\langle  x,y\rangle$. Let
$\pi:Q\rightarrow X^{[2]}$ be the pull-back of the natural
$\mathbb{P}^1$-bundle on $G(1,n+1)$. Let $\alpha:Q\rightarrow
\mathbb{P}^{n+1}$ be the natural map. Then $\alpha^{-1}(X)$ is a
reducible divisor in $Q$, which is generically of degree $3$ over
$X^{[2]}$, with one component $D_1$ which is finite of degree $2$
over $X^{[2]}$, parameterizing the pairs $(x,x+y)$ and a second
component $\mathcal{D}$ which is of degree $1$ over $X^{[2]}$ and
parameterizes generically  the pairs $(z,x+y)$. The divisor
$\mathcal{D}$ is isomorphic to $X^{[2]}$ away from $P_2$. Over
$P_2$, the restricted $\mathbb{P}^1$-bundle $Q_{P_2}$ is contained
in $\alpha^{-1}(X)$ but not in $D_1$, so it is contained in
$\mathcal{D}$. We claim that $\mathcal{D}$ is smooth and identifies
with the blow-up of $X^{[2]}$ along $P_2$. Indeed, this simply
follows from the fact that the divisor $\mathcal{D}$ is the zero set
of a section $s$ of $\mathcal{O}_Q(3)(-D_1)$ on $Q$. The line bundle
$\mathcal{O}_Q(3)(-D_1)$ on $Q$ has degree $1$ along the fibers of
$Q\rightarrow X^{[2]}$, so $R^0\pi_*\mathcal{O}_Q(3)(-D_1)$ is a
rank $2$ vector bundle $\mathcal{E}$ on $X^{[2]}$ such that
$Q=\mathbb{P}(\mathcal{E})$. The section $s$ provides a section $s'$
of $\mathcal{E}$ and one easily checks that $P_2\subset X^{[2]}$ is
scheme-theoretically defined as the zero-locus of  $s'$. This
implies that $\mathcal{D}$ is the blow-up of $X^{[2]}$ along $P_2$
and the smoothness of $P_2$  implies the smoothness of
$\mathcal{D}$. The claim is thus proved. On the other hand, the
rational map $\Phi$ clearly pulls back to a morphism on
$\mathcal{D}$, so (i) is proved.

(ii) and (iii) If the length $2$ subscheme $Z=x+y$ (or $Z=(2x, v)$
with $v$ tangent to $X$ at $x$) does not belong to $P_2$, then the
line $l_{x,y}$ (or $l_{x,v}$) is not contained in $X$, the morphism
$\Phi$ is well defined at  $Z$ and its image is a pair $([l],u)$
where $l$ is a line passing through $u$ and is not contained in $X$.
At such point $([l],u)$ of $P_X$, $\Phi^{-1}$ is well-defined, and
associates to $([l'],u')$ in a neighborhood of $([l],u)$ in $P_X$,
the residual scheme of $u'$ in $l'\cap X$. This proves (iii). It
remains to understand what happens along the exceptional divisor
$Q_{P_2}$ of $\mathcal{D}$. Now we have
$Q_{P_2}=\{(u,x+y,[l]),\,l\subset X, x+y\in l^{(2)},\,u\in l\}$. By
definition, $\widetilde{\Phi}$ maps such a triple to the pair
$(u,[l])$, which by definition belongs to $P_X$. Furthermore, the
fiber of $\widetilde{\Phi}$ over $(u,[l])$ when $l\subset X$, that
is when $(u,[l])\in P$, identifies with
 the plane $l^{(2)}\cong
\mathbb{P}^2$. Thus $\widetilde{\Phi}^{-1}(P)$ is equal to the
smooth irreducible hypersurface $Q_{P_2}$ in the smooth variety
$\mathcal{D}$ and this implies that $\widetilde{\Phi}$ factors
through a morphism $f:\mathcal{D}\rightarrow \widetilde{P_X}$ which
has to be an isomorphism, since it cannot contract any curve ;
indeed, otherwise a contracted curve would be a curve in a fiber
$\mathbb{P}^2$
 as described above, so the whole corresponding
$\mathbb{P}^2$ would be contracted by $f$, hence also all
deformations of this $\mathbb{P}^2$ in
$\mathcal{D}=\widetilde{X^{[2]}}$. But then the divisor $Q_{P_2}$
would be contracted by $f$, while its image has to be the
exceptional divisor of $\widetilde{P_X}\rightarrow P_X$.
\end{proof}
We now first give the proof of Theorem \ref{theoeqchcohdec} in the case
of  cubics of dimension $\leq 4$, because the argument is shorter in this case.
\begin{proof}[Proof of Theorem \ref{theoeqchcohdec} in the case $n\leq4$]
Let  $X$ be a smooth cubic hypersurface  of dimension $\leq 4$ and
assume $X$ admits a cohomological decomposition of the diagonal. The
assumptions of Proposition \ref{leGamma} are satisfied by $X$, since
the integral cohomology of a smooth cubic hypersurface has no
torsion and the integral Hodge conjecture is proved in
\cite{voisinaspects} for cubic fourfolds. Using the notation
introduced previously, there exists by Proposition \ref{leGamma} a
cycle $\Gamma\in {\rm CH}^n(X^{[2]})$ such that
\begin{eqnarray}
\label{eqnew1}\mu^*\Gamma= \Delta_X-x\times X-X\times x-W\,\,{\rm
in} \,\,{\rm CH}^n(X\times X),
\end{eqnarray}
 with $W$ supported over
$D\times X$, $D\subsetneqq X$, and  $[\Gamma]=0$ in
$H^{2n}(X^{[2]},\mathbb{Z})$.

 By
Proposition \ref{propX2}, the blow-up
$\sigma:\widetilde{X^{[2]}}\rightarrow X^{[2]}$ of $X^{[2]} $ along
$P_2$ identifies via $\widetilde{\Phi}$ with a blow-up of the
projective bundle $P_X$ over $X$. Furthermore, the exceptional
divisor  of $\widetilde{\Phi}:\widetilde{X^{[2]}}\rightarrow P_X$ is
also the exceptional divisor of
$\sigma:\widetilde{X^{[2]}}\rightarrow X^{[2]}$, hence maps via
$\sigma$ to $P_2\subset X^{[2]}$. It follows that the pull-back
$\sigma^*(\Gamma)$ of the cycle $\Gamma$ to $\widetilde{X^{[2]}}$
can be written as
\begin{eqnarray} \label{eqprooftheo1}\sigma^*(\Gamma)=\Gamma_1+\Gamma_2,
\end{eqnarray}
where $\Gamma_1$ and $\Gamma_2$ are cohomologous to $0$, $\Gamma_1$
is a cycle cohomologous to $0$ on the exceptional divisor of
$\widetilde{\Phi}:\widetilde{X^{[2]}}\rightarrow P_X$, and
$\Gamma_2$ is the pull-back of a cycle $\Gamma'_2$ cohomologous to
$0$ on $P_X$. As the exceptional divisor of $\widetilde{\Phi}$
equals the exceptional divisor of $\sigma$, it follows from
(\ref{eqprooftheo1}), by applying $\sigma_*$, that
\begin{eqnarray} \label{eqprooftheo2}\Gamma={i_{P_2}}_*(\Gamma'_1)+\Phi^*(\Gamma'_2)\,\,{\rm in}\,\,{\rm CH}(X^{[2]}),
\end{eqnarray}
where $\Gamma'_1$ is a cycle  cohomologous to $0$ on $P_2$. Here
$i_{P_2}$ denotes the inclusion  map of $P_2$ in $X^{[2]}$. It is
known that for a smooth cubic hypersurface of dimension $\leq 4$,
cycles homologous to $0$ are algebraically equivalent to $0$. For
codimension $2$ cycles, this is proved by Bloch and Srinivas
\cite{blochsrinivas} and is true more generally for any rationally
connected variety;  for $1$-cycles on cubic fourfolds, this is
proved in \cite{zong}, and this is true more generally for
$1$-cycles on Fano complete intersections of index $\geq2$. The
result then also holds for cycles on a projective bundle over a
cubic of dimension $\leq4$. Thus $\Gamma'_2$ is algebraically
equivalent to $0$ and thus we conclude that
\begin{eqnarray} \label{eqprooftheo3}\Gamma={i_{P_2}}_*(\Gamma'_1)\,\,{\rm in}\,\,{\rm CH}(X^{[2]})/{\rm alg},
\end{eqnarray}
for some $n$-cycle $\Gamma'_1$ homologous to $0$  on $P_2$. We apply
now the following result:
\begin{lemm} \label{leP2} Let $X$ be a $n$-dimensional smooth cubic
hypersurface and $Z$ be a $n$-cycle homologous to $0$ on
$P_2\stackrel{i_{P_2}}{\hookrightarrow} X^{[2]}$. Then
$\mu^*({i_{P_2}}_*Z)\in {\rm CH}^n(X\times X)$ is supported on
$D\times X$ for some proper closed algebraic subset $D$ of $X$.
\end{lemm}
\begin{proof}
Recall that $P_2$ is the union of the symmetric products $L^{(2)}$
over all lines $L\subset X$.  As before, denote by
$$q:P\rightarrow X,\,\,p:P\rightarrow F$$
 the natural maps, and by
 $q_2$ the natural map
$P\times_FP\rightarrow X\times X$ induced by $q$. We will also
denote by $\pi:P\times_FP\rightarrow F$ the map induced by $p$. Via
$\pi$, $P\times_FP$ is a $\mathbb{P}^1\times \mathbb{P}^1$-bundle
over $F$. Let $H:=c_1(\mathcal{O}_X(1))\in {\rm CH}^1(X)$ and let
$h=q^*H\in {\rm CH}^1(P)$ be its pull-back to $P$. For a cycle $Z$
supported on $P_2$, we have
 $$\mu^*({i_{P_2}}_*Z)=q_{2*}({r'}^*Z)\,\,{\rm in}\,\,{\rm CH}(X\times X)$$
 where  $r'$ is
the quotient map $P\times_FP\rightarrow P_2$. Let $T:={r'}^*Z\in
{\rm CH}(P\times_FP)$. This is a cycle  homologous to $0$ on
$P\times_F P$, and thus it  can be written as
\begin{eqnarray} \label{eqprooftheo5}T=h_1h_2\pi^*\alpha+h_1\pi^*\beta+h_2\pi^*\gamma+\pi^*\zeta
\,\,{\rm in}\,\,{\rm CH}(P\times_F P),
\end{eqnarray}
for some cycles $\alpha,\,\beta,\,\gamma,\,\zeta$ homologous to $0$
on $F$, with $h_i=pr_i^*h$,
 for $i=1,\,2$,  $pr_i:P\times_FP\rightarrow P$ being the
$i$-th projection.

We now push-forward these cycles to $X\times X$ via $q_2$
 and observe that the three cycles
 $$q_{2*}(\pi^*\zeta),\,q_{2*}(h_1h_2\pi^*\alpha),\,q_{2*}(h_1\pi^*\beta)$$ are cycles
 supported on $D\times X$ for some $D\subsetneqq X$. Indeed, for the two last ones, this is due to the projection
 formula (and the equality
 $h_1=q_2^*(H_1)$, where $H_1:=pr_1^*H\in {\rm CH}^1(X\times X)$), and for the first one,  this is because $q_{2*}(\pi^*\zeta)$ is supported on
 $q(p^{-1}({\rm Supp}\,\zeta))\times X$.
 Now $\zeta$ is a $(n-2)$-cycle, so $q(p^{-1}({\rm Supp}\,\zeta))$ is a proper closed
 algebraic subset of $X$.
It remains to  examine the cycle $q_{2*}(h_2\pi^*\gamma)$. We
observe now that the diagonal $\Delta_P\subset P\times_FP\subset
P\times P$ is a divisor $d$ in $P\times_FP$ whose class is of the
form $d=h_1+h_2+\pi^*\lambda$ for some divisor class $\lambda\in
{\rm CH}^1(F)$. Furthermore, we obviously have $q_2(\Delta_P)\subset
\Delta_X$. Thus we can write
$$h_2\pi^*\gamma=(d-h_1-\pi^*\lambda)\pi^*\gamma$$
in ${\rm CH}(P\times_FP)$.
We thus have
$$q_{2*}(h_2\pi^*\gamma)=q_{2*}(d\pi^*\gamma)-q_{2*}(h_1\pi^*\gamma)-q_{2*}(\pi^*(\lambda\gamma))\,\,{\rm in}
\,\,{\rm CH}(X\times X).$$ As already explained, the cycles
$q_{2*}(h_1\pi^*\gamma),\,q_{2*}(\pi^*(\lambda\gamma))$ are
supported over $D\times X$ for some $D\subsetneqq X$. Finally, the
last cycle $q_{2*}(d\pi^*\gamma)$ has to be $0$ in ${\rm CH}(X\times
X)$. Indeed, this is a $n$-cycle of $X\times X$ which is supported
on the diagonal, hence proportional to it, and also cohomologous to
$0$.
\end{proof}
Combining (\ref{eqnew1}), (\ref{eqprooftheo3}) and Lemma \ref{leP2}, we conclude that
\begin{eqnarray} \label{eqprooftheo6} \Delta_X=x\times X+X\times x-W'\,\,{\rm in}\,\,{\rm CH}(X\times X)/{\rm alg},
\end{eqnarray}
where $W'$ is supported on $D'\times X$ for some $D'\subsetneqq X$.
In conclusion, $X$ admits a decomposition of the diagonal modulo
algebraic equivalence, and we can now apply Proposition
\ref{propredalg} to conclude that $X$ admits a Chow-theoretic
decomposition of the diagonal.
\end{proof}
We now turn to the general case.
\begin{proof}[Proof of Theorem \ref{theoeqchcohdec} for general   $n$] Let $X$ be a smooth cubic hypersurface such that the group
$H^{2*}(X,\mathbb{Z})/H^{2*}(X,\mathbb{Z})_{alg}$ has no
$2$-torsion. Then  Propositions \ref{leGamma} and \ref{propX2} apply.
Next, if we examine the proof given in the case of dimension $\leq4$, we see that the
only  place where we used the fact that ${\rm dim}\,X\leq 4$ is
in the analysis of the term $\Phi^*(\Gamma'_2)$, where $\Gamma'_2$
is cohomologous to $0$ on $P_X$. In the above proof, we directly
used the fact that this term is algebraically equivalent to $0$,
which we do not know in higher dimension. The following provides  an
alternative argument which works also in higher dimension:
\begin{lemm}Let $X$ be a smooth cubic hypersurface
of dimension   $n\geq 2$. Let $Z$ be a $n$-cycle cohomologous to $0$
on $P_X$. Then $3\mu^*(\Phi^*(Z))\in {\rm CH}^n(X\times X)$ is
rationally equivalent to a cycle supported on $D\times X$, for some
$D\subsetneqq X$.
\end{lemm}
\begin{proof} Recall that $p_X:P_X\rightarrow X$ is the
$\mathbb{P}^n$-bundle over $X$ with fiber over $x\in X$ the set of
lines in $\mathbb{P}^{n+1}$ through $x$. Let us denote by $l\in {\rm
CH}^1(P_X)$ the class of $\mathcal{O}_{P_X}(1)$ (we  choose for
$\mathcal{O}_{P_X}(1)$ the pull-back of the Pl\"ucker line bundle on
the Grassmannian of lines $G(1,\mathbb{P}^{n+1})$). The cycle $Z$
can be written as
\begin{eqnarray}\label{eqcycle27juin}
Z=p_X^*Z_0+lp_X^*Z_1+\ldots+l^{n}p_X^* Z_n, \end{eqnarray} where
$Z_i$ are cycles of codimension $n-i$ on $X$. Note that the cycles
$Z_i$ are all homologous to $0$. As ${\rm dim}\,X\geq 2$, ${\rm
CH}_0(X)_{hom}=0$ and thus $Z_0=0$ in ${\rm CH}_0(X)$. Hence
(\ref{eqcycle27juin}) shows that $Z=l\cdot Z'$ for some $Z'\in {\rm
CH}(P_X)$. Let $\Psi:=\Phi\circ \mu:X\times X\dashrightarrow P_X$
and
 let $\widetilde{\Psi}: \widetilde{\widetilde{X\times X}}\rightarrow
P_X$ be the desingularization of $\Psi$ obtained by blowing-up first
the diagonal of $X$, and then the inverse image of $P_2\subset
X^{[2]}$ (see Proposition \ref{propX2}). Let
$\tilde{\tau}:\widetilde{\widetilde{X\times X}}\rightarrow X\times
X$ be the composition of the two blow-ups. There are two exceptional
divisors of $\tilde{\tau}$, namely $E_\Delta$ and $E_{P_2}$. We thus
have (using the fact that $\widetilde{\Psi}$ factors through
$X^{[2]}$)
$$\widetilde{\Psi}^*(l)=\alpha \tilde{\tau}^*(H_1+H_2)+\beta E_{\Delta}+\gamma
E_{P_2},$$ where the explicit computation of the coefficients
$\alpha,\,\beta,\,\gamma$ can be done but is not useful here.

It follows that
\begin{eqnarray}\label{eqcycle27juin2}
\Psi^*(Z)=\Psi^*(l\cdot Z')=\tilde{\tau}_*(\widetilde{\Psi}^*(l\cdot
Z'))=
\tilde{\tau}_*(\widetilde{\Psi}^*(l)\cdot \widetilde{\Psi}^*(Z'))\\
\nonumber = \tilde{\tau}_*((\alpha \tilde{\tau}^*(H_1+H_2)+\beta
E_{\Delta}+\gamma E_{P_2})\cdot \widetilde{\Psi}^*(Z'))\,\,{\rm
in}\,\,{\rm CH}(X\times X).
\end{eqnarray}
Now, we develop the last expression and  observe again that
$\tilde{\tau}_*(\beta E_{\Delta}\cdot \widetilde{\Psi}^*(Z'))$ is
supported on the diagonal of $X$, it must be proportional to
$\Delta_X$, hence in fact identically $0$ as it is cohomologous to
$0$. Next, the cycle $\tilde{\tau}_*(\gamma E_{P_2}\cdot
\widetilde{\Psi}^*(Z'))$ comes from a  $n$-cycle $Z''$ homologous to
$0$ on $P\times_F P$, i.e.
\begin{eqnarray}\label{eqcycle27juin3}
\tilde{\tau}_*(\gamma E_{P_2}\cdot
\widetilde{\Psi}^*(Z'))={q_2}_*(Z''),
\end{eqnarray}
where $q_2:P\times_FP\rightarrow X\times X$ is introduced above. We
can then apply Lemma \ref{leP2} to conclude that the cycle
$\tilde{\tau}_*(\gamma E_{P_2}\cdot \widetilde{\Psi}^*(Z'))$ is
supported on $D\times X$ for some $D\subsetneqq X$.
  Thus we conclude from (\ref{eqcycle27juin2}), the
projection formula, and the analysis above that
\begin{eqnarray}\label{eqcycle27juin5}
\Psi^*(Z)=H_1\cdot W_1+H_2\cdot W_2+W \,\,{\rm in}\,\,{\rm
CH}(X\times X),
\end{eqnarray}
where $W$ is  supported on $D\times X$ for some $D\subsetneqq X$,
and $W_1,\,W_2$ are cycles homologous  to $0$ on $X\times X$. It
thus suffices to show that cycles $\Gamma$ on $X\times X$ of the
form $H_1\cdot W_1$ and $H_2\cdot W_2$ with $W_i$ homologous to $0$
on $X\times X$ have the property that $3\Gamma$ is rationally
equivalent to a cycle supported on $D\times X$ for some proper
closed algebraic subset of $X$. For $H_1\cdot W_1$ this is obvious,
and the coefficient $3$ is not needed. For $H_2\cdot W_2$, we
observe that if $i_2:X\times X\rightarrow X\times \mathbb{P}^{n+1}$
denotes the natural inclusion, we have
$$3H_2\cdot W_2=i_2^*\circ i_{2*}(W_2)\,\,{\rm in}\,\,{\rm
CH}(X\times X).$$ Now the cycle $i_{2*}(W_2)\in {\rm
CH}_{n+1}(X\times \mathbb{P}^{n+1})$ is homologous to $0$. Using its
decomposition $$ i_{2*}(W_2)=\sum_i pr_1^*\gamma_i\cdot
pr_2^*L^{n-i},$$ with $\gamma_i\in {\rm CH}^i(X)_{hom}$ and
$L=c_1(\mathcal{O}_{\mathbb{P}^{n+1}}(1))\in {\rm
CH}^1(\mathbb{P}^{n+1})$, we thus conclude that $\gamma_0=0$, hence
that $i_{2*}(W_2)=\sum_{i>0} pr_1^*\gamma_i\cdot pr_2^*L^{n-i}$ in
${\rm CH}^n(X\times \mathbb{P}^{n+1})$. Hence $i_{2*}(W_2)$ is
rationally equivalent to a cycle supported on $D\times
\mathbb{P}^{n+1}$ for some proper closed algebraic subset $D$ of
$X$, and thus $3H_2\cdot W_2=i_2^*\circ i_{2*}(W_2)$ is rationally
equivalent to a cycle supported on $D\times X$ for some proper
closed algebraic subset $D$ of $X$.
\end{proof}
The rest of the proof goes as before, using again  Lemma \ref{leP2},
and this allows to conclude that, if $X$ admits a cohomological
decomposition of the diagonal, then $ 3(\Delta_X-X\times x)$ is
rationally equivalent to a cycle supported on $D\times X$, for some
$D\subsetneqq X$. On the other hand, as $X$ admits a unirational
parametrization of degree $2$ (see \cite{clemensgriffiths}), we also
know that $ 2(\Delta_X-X\times x)$ is rationally equivalent to a
cycle supported on $D'\times X$ , for some $D'\subsetneqq X$. It
follows  that $X$ admits a Chow-theoretic decomposition of the
diagonal.
\end{proof}
\section{Criteria for the cohomological decomposition of the diagonal
\label{sec3}} This section is devoted to the existence  of
cohomological decomposition of the diagonal. Our main result here
is the following criterion for such a decomposition to exist:

\begin{theo} \label{theodeccohdutext} (cf. Theorem \ref{theocohdecomp}) Let $X$ be smooth
projective of dimension $n$. If $X$ admits a cohomological
decomposition of the diagonal,
 then the following condition (*)
is satisfied:

(*) There exist smooth projective varieties $Z_i$ of dimension $n-2$,  correspondences
$\Gamma_i\in {\rm CH}^{n-1}(Z_i\times X)$, and integers
$n_i$, such that for any
$\alpha,\,\beta\in H^n(X,\mathbb{Z})$,
\begin{eqnarray} \label{eqint10ju}\langle  \alpha,\beta\rangle_X=\sum_in_i\langle  \Gamma_i^*\alpha,\Gamma_i^*\beta\rangle_{Z_i}.
\end{eqnarray}
Conversely,  assume condition (*) and furthermore

(i) $H^{2i}(X,\mathbb{Z})$ is algebraic for $2i\not=n$ and
$H^{2i+1}(X,\mathbb{Z})=0$ for $2i+1\not=n$.

(ii) $H^*(X,\mathbb{Z})$ has no torsion.

Then $X$ admits a cohomological decomposition of the diagonal.
\end{theo}
Here $\Gamma_i^*:H^n(X,\mathbb{Z})\rightarrow H^{n-2}(Z_i,\mathbb{Z})$ is as usual defined by
$$\Gamma_i^*(\alpha):={pr_{Z_i}}_*(pr_X^*\alpha\smile [\Gamma_i]),$$
where $pr_{Z_i},\,pr_X$ are the projections from $Z_i\times X$ to
its factors. Let us comment on assumptions (i) and (ii). If $X$ is a
 complete intersection of dimension $n$, the integral cohomology
of $X$ has no torsion and the groups $H^{2i}(X,\mathbb{Z})$ are
cyclic generated by $h^i$, $h=c_1(\mathcal{O}_X(1))$ for $i<  n/2$.
For $i>n/2$, they are also cyclic but the generator is now
$\frac{1}{d}h^i$, where $d={\rm deg}\,X$ and it is not true in
general that they are generated by a cycle class, except when $X$ is
Fano and  $n=4$ (resp. $n=3$), in which case  $H^6(X,\mathbb{Z})$
(resp. $H^4(X,\mathbb{Z})$) is generated by the class of a line in
$X$, and some sporadic cases. Note that in any case the class $ h^i$
is algebraic for any $i$, and in some cases this can be used as a
substitute assumption in Theorem \ref{theodeccohdutext}, like smooth
cubic hypersurfaces (see Corollary
 \ref{rema18juill}).
Another interesting class of varieties  which satisfy these two
properties, needed  in order to apply Theorem \ref{theodeccohdutext}
below, is the class of rationally connected threefolds with trivial
Artin-Mumford invariant, for which it is proved in
\cite{voisinuniruled} that $H^4(X,\mathbb{Z})$ is algebraic. In this
case, one gets Theorem \ref{cororc3} which improves \cite[Corollary
4.5 and Theorem 4.9]{voisinjag}.
\begin{proof}[Proof of Theorem \ref{theodeccohdutext}] Let us first prove that assuming (i) and (ii),
  condition (*) implies that
$X$ admits a cohomological decomposition of the diagonal. So let
$Z_i,\,\Gamma_i$ be as above and satisfy (\ref{eqint10ju}). As ${\rm
dim}\,Z_i=n-2$, and ${\rm codim}\,\Gamma_i=n-1$, the $\Gamma_i$'s
are $(n-1)$-cycles in $Z_i\times X$. We denote by
$(\Gamma_i,\Gamma_i)\in {\rm CH}^{2n-2}(Z_i\times Z_i\times X\times
X)$ the correspondence $p_{13}^*\Gamma_i\cdot p_{24}^*\Gamma_i$
between $Z_i\times Z_i$ and $X\times X$, where the $p_{rs}$ are the
projectors from $Z_i\times Z_i\times X\times X$ to the product of
two of its factors. Observe that $(\Gamma_i,\Gamma_i)_*\Delta_{Z_i}$
is supported on $D_i\times D_i$, where $D_i\subsetneqq X$ is defined
as the image of ${\rm Supp}\,\Gamma_i$ in $X$ by the second
projection. Let
$$\Gamma:=\sum_in_i (\Gamma_i,\Gamma_i)_*\Delta_{Z_i}\in {\rm CH}^n(X\times X).$$
Equation (\ref{eqint10ju}) can be written as
\begin{eqnarray} \label{eqint10ju1}\langle  \alpha,\beta\rangle_X=\int_{X\times X} pr_1^*\alpha\smile pr_2^*\beta\smile [\Gamma]
\end{eqnarray}
for any degree $n$ classes $\alpha,\beta$ on $X$.
It follows that the class
\begin{eqnarray} \label{eqint10ju2}[\Delta_X]-[\Gamma]\in H^{2n}(X\times X,\mathbb{Z})\cong
 {\rm End}_0\,(H^*(X,\mathbb{Z}))
\end{eqnarray}
annihilates  $H^n(X,\mathbb{Z})$. In (\ref{eqint10ju2}), ${\rm
End}_0$ denotes the group of degree preserving endomorphisms. The
isomorphism $H^{2n}(X\times X,\mathbb{Z})\cong {\rm
End}_0\,(H^*(X,\mathbb{Z}))$ is a consequence, by K\"unneth
decomposition and Poincar\'e duality, of the fact that
$H^*(X,\mathbb{Z})$ has no torsion.

It follows that we have (again by K\"unneth decomposition)
\begin{eqnarray} \label{eqint10ju3}[\Delta_X]-[\Gamma]\in \oplus_{i\not=n} H^i(X,\mathbb{Z})\otimes H^{2n-i}(X,\mathbb{Z})
\subset H^{2n}(X\times X,\mathbb{Z}).
\end{eqnarray}
On the other hand, condition (i) tells us that $H^{*\not=n}(X,\mathbb{Z})$ consists of classes of algebraic cycles, so that (\ref{eqint10ju3}) becomes

\begin{eqnarray} \label{eqint10ju4}[\Delta_X]-[\Gamma]=\sum_i  pr_1^*[W_i]\smile pr_2^* [W'_i]
\end{eqnarray}
for some cycles $W_i,\,W'_i$ of $X$ with ${\rm dim}\,W_i+{\rm
dim}\,W'_i=n$. The right-hand side of (\ref{eqint10ju4}) is of the
form
$$[X\times x]+\sum_{i,\,{\rm dim}\,W'_i>0}  [pr_1^*W_i\cdot  pr_2^* W'_i]$$ and
clearly $\sum_{i,\,{\rm dim}\,W'_i>0}  pr_1^*W_i\cdot  pr_2^* W'_i$
is supported on $D'\times X$ for some proper closed algebraic subset
$D'\subsetneqq X$. Hence we get
$$[\Delta_X]-[X\times x]=[\Gamma]+\sum_{i,\,{\rm dim}\,W'_i>0}  [pr_1^*W_i\cdot  pr_2^* W'_i]=[Z],$$
where the cycle $Z=\Gamma+\sum_{i,\,{\rm dim}\,W'_i>0}
pr_1^*W_i\cdot pr_2^* W'_i$
 is supported on $(\cup D_i\cup D')\times X$.

We now  prove conversely that  condition (*) is implied by the existence
of  a cohomological decomposition of the diagonal of $X$.
Let $D\subset X$ be a divisor and
$Z\subset D\times X$ be an $n$-cycle such that
$$ [Z]=[\Delta_X]-[X\times x]\,\,{\rm in}\,\,H^{2n}(X\times X,\mathbb{Z}).$$
We first claim that we can assume that $D$ is a normal crossing
divisor.  In order to achieve this, let $\tau:X'\rightarrow X$ be a
blow-up of $X$ such that a global normal crossing divisor $D'\subset
X'$ dominates $D$. Enlarging $D$ if necessary, we can assume that
$Z$ lifts to a $n$-cycle $Z'\in {\rm CH}^n(D''\times X')$, where
$D''=\sqcup_iD_i$ is the normalization of $D'$. (Indeed, it suffices
to choose $D$ in such a way that for each irreducible component
$Z_i$ of the support of $Z$, $D$  has at least one component which
is generically smooth along $Z_i$.)
 It follows easily that the cycle class
$$[\Delta_{X'}]-[X'\times x]\,\,\in\,\,H^{2n}(X'\times X')$$
is the class of a cycle $Z_1$ supported on $(D'\cup E)\times X'$,
where $E$ is the exceptional divisor of $\tau$. The divisor
$D'_1=D'\cup E$ can be also assumed to have global normal crossings
and thus the cycle $Z_1$ lifts to  a $n$-cycle $Z'_1\in {\rm
CH}^n(D''_1\times X')$, where $D''_1=\sqcup_iD_{1,i}$ is the
normalization of $D'_1$. On the other hand, as we
 have
 $\langle  \alpha,\beta\rangle_X=\langle  \tau^*\alpha,\tau^*\beta\rangle_{X'}$ for $\alpha,\,\beta\in H^n(X,\mathbb{Z})$,
 it suffices to prove (*) for $X'$.
 This proves our claim.

 From now on, we thus assume $X=X'$ and
 $D$ is a global normal crossing divisor in $X$ with normalization
 $\sqcup_iD_i$, so that $Z$ lifts to a cycle $\widetilde{Z}$ in $(\sqcup_iD_i)\times X$.
Let us denote by $\Gamma_i\in {\rm CH}^n(D_i\times X)$ the
restriction of $\widetilde{Z}$ to the connected component $D_i\times
X$. Let $k_i:D_i\rightarrow X$ be the inclusion map. We have
\begin{eqnarray}\label{eqpreuvenec1}
\sum_i(k_i,Id_X)_*[\Gamma_i]=[Z]=[\Delta_X]-[X\times x]\,\,{\rm
in}\,\,H^{2n}(X\times X).
\end{eqnarray}
We can of course  assume $n>0$, so that $[Z]^*\alpha=\alpha$ for any
$\alpha\in H^n(X,\mathbb{Z})$. Then  (\ref{eqpreuvenec1})  gives the
following equality, for any $\alpha,\,\beta\in H^n(X,\mathbb{Z})$
\begin{eqnarray}\label{eqpreuvenec2}
 \langle \alpha,\beta\rangle_X=\langle
[Z]^*\alpha,[Z]^*\beta\rangle_X  =\langle
\sum_i((k_i,Id_X)_*[\Gamma_i])^*\alpha,\sum_i
((k_i,Id_X)_*[\Gamma_i])^*\beta\rangle_X.
\end{eqnarray}
We now develop the last expression, which gives for all
$ \alpha,\,\beta\in H^n(X,\mathbb{Z})$:
\begin{eqnarray}\label{eqpreuvenec3}
\langle\alpha,\beta\rangle_X=\sum_{i,j}\langle
((k_i,Id_X)_*[\Gamma_i])^*\alpha,
((k_j,Id_X)_*[\Gamma_j])^*\beta\rangle_X.
\end{eqnarray}
Note now that
$$((k_i,Id_X)_*[\Gamma_i])^*\alpha=k_{i*}([\Gamma_i]^*\alpha)\,\,{\rm in}\,\,H^n(X,\mathbb{Z})$$
and similarly for $(k_j,Id_X)_*[\Gamma_j])_*\beta$. Hence
(\ref{eqpreuvenec3}) becomes
\begin{eqnarray}\label{eqpreuvenec4}
\langle \alpha,\beta\rangle_X=\sum_{i,j}\langle
k_{i*}([\Gamma_i]^*\alpha), k_{j*}([\Gamma_j]^*\beta)\rangle_X,
\end{eqnarray}
where $[\Gamma_i]^*\alpha\in H^{n-2}(D_i,\mathbb{Z})$ and similarly
$[\Gamma_j]^*\beta\in H^{n-2}(D_j,\mathbb{Z})$. Let
$\delta_i=k_i^*(D_i)\in {\rm CH}^1(D_i)$ and $[\delta_i]\in
H^2(D_i,\mathbb{Z})$ its cohomology class. As $k_i$ is an embedding,
we have $k_i^*\circ k_{i*}=[\delta_i]\smile:
H^{n-2}(D_i,\mathbb{Z})\rightarrow H^{n}(D_i,\mathbb{Z})$, and thus
\begin{eqnarray}\label{eqpreuvenec5} \langle k_{i*}([\Gamma_i]^*\alpha),
k_{i*}([\Gamma_i]^*\beta)\rangle_X=\langle [\delta_i]\smile
[\Gamma_i]^*\alpha,[\Gamma_i]^*\beta\rangle_{D_i}.
\end{eqnarray}
Write $\delta_i=\sum_{l}n_{il} Z_{il}$ where $n_{il}\in \mathbb{Z}$
and $Z_{il}$ is a smooth $(n-2)$-dimensional subvariety of $D_i$.
Then letting $\Gamma_{il}\in {\rm CH}^{n-1}(Z_{il}\times X)$ be the
pull-back of $\Gamma_i$ to $Z_{il}\times X$, (\ref{eqpreuvenec5})
can be written as
\begin{eqnarray}\label{eqpreuvenec6} \langle  k_{i*}([\Gamma_i]^*\alpha),
k_{i*}([\Gamma_i]^*\beta)\rangle_X= \sum_ln_{il}\langle
[\Gamma_{il}]^*\alpha,[\Gamma_{il}]^*\beta\rangle_{Z_{il}}.
\end{eqnarray}
The right hand side of this equation is exactly of the form allowed in
 (\ref{eqint10ju}) and it remains to analyze in (\ref{eqpreuvenec4}) the terms
$$\langle  k_{i*}([\Gamma_i]^*\alpha),
k_{j*}([\Gamma_j]^*\beta)\rangle_X+\langle
k_{j*}([\Gamma_j]^*\alpha), k_{i*}([\Gamma_i]^*\beta)\rangle_X$$ for
$i\not=j$. Denote by $W_{ij}$ the intersection $D_i\cap D_j$. It
admits two correspondences $\Gamma_{ij},\,\Gamma_{ji}\in {\rm
CH}^{n-1}(W_{ij}\times X)$, namely the restriction to $W_{ij}\times
X$ of $\Gamma_i\in {\rm CH}^{n-1}(D_i\times X)$ and the restriction
to $W_{ij}\times X$ of $\Gamma_j \in {\rm CH}^{n-1}(D_j\times X)$
respectively. With this notation, we have
\begin{eqnarray}\label{eqpreuvenec7} \langle  k_{i*}([\Gamma_i]^*\alpha),
k_{j*}([\Gamma_j]^*\beta)\rangle_X=\langle
[\Gamma_{ij}]^*\alpha,[\Gamma_{ji}]^*\beta\rangle_{W_{ij}},
\\
\nonumber \langle  k_{j*}([\Gamma_j]^*\alpha),
k_{i*}([\Gamma_i]^*\beta)\rangle_X=\langle
[\Gamma_{ji}]^*\alpha,[\Gamma_{ij}]^*\beta\rangle_{W_{ij}},
\end{eqnarray}

which provides
\begin{eqnarray}\label{eqpreuvenec8}\langle  k_{i*}([\Gamma_i]^*\alpha),
k_{j*}([\Gamma_j]^*\beta)\rangle_X+\langle
k_{j*}([\Gamma_j]^*\alpha), k_{i*}([\Gamma_i]^*\beta)\rangle_X
\\
\nonumber =\langle
[\Gamma_{ij}]^*\alpha,[\Gamma_{ji}]^*\beta\rangle_{W_{ij}}+
\langle  [\Gamma_{ji}]^*\alpha,[\Gamma_{ij}]^*\beta\rangle_{W_{ij}}\\
\nonumber =\langle
([\Gamma_{ij}]+[\Gamma_{ji}])^*\alpha,([\Gamma_{ij}]+[\Gamma_{ji}])^*\beta\rangle_{W_{ij}}-
\langle [\Gamma_{ij}]^*\alpha,[\Gamma_{ij}]^*\beta\rangle_{W_{ij}}
\\
\nonumber -\langle  [\Gamma_{ji}]^*\alpha,
[\Gamma_{ji}]^*\beta\rangle_{W_{ij}}.
\end{eqnarray}
Each of the terms appearing in the final expression of
(\ref{eqpreuvenec8}) is of the form allowed in
 (\ref{eqint10ju}), which concludes the proof.
\end{proof}
\begin{coro}\label{rema18juill} Let $X$ be a smooth cubic hypersurface. Then
$X$ admits a cohomological decomposition of the diagonal, (or
equivalently a Chow-theoretic one  if
$H^*(X,\mathbb{Z})/H^*(X,\mathbb{Z})_{alg}$ has no torsion,) if and
only if $X$ satisfies condition (*).
\end{coro}
\begin{proof} The necessity of the condition
(*) is proved above and the proof does not use the assumptions (i)
and (ii) of  Theorem \ref{theodeccohdutext}, so it works in our
case.
 For the converse, it is in fact not a direct corollary of the theorem, since the proof
  uses these assumptions and we do not know
that cubic hypersurfaces satisfy  assumption (ii), but  we can make
a small variant of the proof, using the following observation: As a
smooth cubic hypersurface of dimension $\geq2$ admits a unirational
parametrization of degree $2$, twice its diagonal admits a
decomposition
$$2\Delta_X=2(X\times x)+Z\,\,{\rm in}\,\,{\rm CH}(X\times X),$$
with $Z$ supported on $D\times X,\,D\subsetneqq X$. So  $X$ admits a
cohomological (or Chow-theoretic) decomposition of the diagonal if
there is  such a decomposition for $3\Delta_X$. But we know that
 $h^i$ is algebraic for any $i$, and thus each class
 $3pr_1^*\alpha_i\smile
  pr_2^*\alpha_{n-i}$ for $2i\not=n,\,0$, is the class of a cycle
  supported on
  $D\times X$, $D\subsetneqq X$, where $\alpha_i$ is a generator of $H^{2i}(X,\mathbb{Z})$.
The proof of the existence of a decomposition of $3[\Delta_X]$
assuming condition (*) then works as  the proof of  Theorem
\ref{theodeccohdutext}.
\end{proof}

Let us conclude this section with the following variant of (part of) Theorem
\ref{theodeccohdutext}.

\begin{theo}\label{theodeccohdutextvariant} Let $X$ be a smooth projective variety of dimension $n$, and let
$N$ be an integer. Assume there is a decomposition
$$N[\Delta_X]=N[X\times x]+[Z]\,\,{\rm in}\,\,H^{2n}(X\times X,\mathbb{Z}),$$
where $Z$ is supported on $D\times X$, $D\subsetneqq X$. Then
there
 exist smooth projective varieties $Z_i$ of dimension $n-2$,  correspondences
$\Gamma_i\in {\rm CH}^{n-1}(Z_i\times X)$, and integers
$n_i$, such that for any
$\alpha,\,\beta\in H^n(X,\mathbb{Z})$,
\begin{eqnarray} \label{eqint10juvar}N^2\langle  \alpha,\beta\rangle_X=\sum_in_i\langle  \Gamma_i^*\alpha,\Gamma_i^*\beta\rangle_{Z_i}.
\end{eqnarray}
\end{theo}
\begin{proof} We look at the proof that condition (*) is implied by the existence of a cohomological
decomposition of the diagonal and we repeat it replacing everywhere
$\langle  \alpha,\beta\rangle_X$  by $N^2\langle
\alpha,\beta\rangle_X$. The main point is the fact that with the
same notation as in the proof of the theorem, letting
$Z=\sum_i(k_i,Id_X)_*\Gamma_i$, we have  by assumption
$$[Z]^*\alpha=N\alpha$$ for $\alpha\in H^n(X,\mathbb{Z})$ (with
$n>0$), and thus
$$N^2\langle  \alpha,\beta\rangle_X=\langle  [Z]^*\alpha,[Z]^*\beta\rangle_X.$$
\end{proof}

\section{Rationally connected  threefolds
\label{sec4}} In the case of rationally connected  threefolds, we
have the following result, which was partially proved in
\cite{voisinjag}:
\begin{theo} \label{cororc3}
(Cf. Theorem \ref{theocohdecomp3fold}) Let $X$ be a rationally
connected $3$-fold and let  $J(X)$ be its intermediate Jacobian,
with principal polarization $\theta\in H^2(J(X),\mathbb{Z})$. Then
$X$ admits a cohomological decomposition of the diagonal if and only
if

(a) $H^3(X,\mathbb{Z})$ has no torsion.

(b) There exists a universal codimension $2$ cycle $\Gamma$ on $J(X)\times X$.

(c)  The integral Hodge  class $\theta^{g-1}/(g-1)!\in
H^{2g-2}(J(X),\mathbb{Z})$, $g:={\rm dim}\,J(X)$, is algebraic
on $J(X)$, that is, is the class of  a $1$-cycle $Z\in{\rm
CH}_1(J(X))$.
\end{theo}
\begin{proof} Indeed, it is proved in \cite[Corollary 4.5]{voisinjag} that for a rationally connected
$3$-fold, conditions (a) and (b) are necessary for the existence of
a cohomological decomposition of the diagonal, and in \cite[Theorem
4.9]{voisinjag} that (a), (b) and (c) are sufficient for the
existence of a cohomological decomposition of the diagonal. So it
suffices to show that (c) is also necessary. Assume $X$ admits a
cohomological decomposition of the diagonal, and let
$Z_i,\,\Gamma_i$ and $n_i$ be as in Theorem \ref{theodeccohdutext}.
Then the Abel-Jacobi map $\Phi_X$ of $X$  induces (after choosing a
reference point in $Z_i$) a morphism
$$\gamma_i=\Phi_X\circ\Gamma_{i*}: Z_i\rightarrow {\rm CH}^2(X)_{hom}\rightarrow J(X)$$
with image $Z'_i:=\gamma_{i*}Z_i\in{\rm CH}_1(J(X))$. We claim that
(\ref{eqint10ju}) is equivalent to the equality
\begin{eqnarray}\label{eqclasscourbe}\sum_in_i[Z'_i]=\theta^{g-1}/(g-1)!.
\end{eqnarray}
Indeed, as $\bigwedge^2H^1(J(X),\mathbb{Z})\cong
H^2(J(X),\mathbb{Z})=H^{2g-2}(J(X),\mathbb{Z})^*$,
(\ref{eqclasscourbe})
is equivalent to the fact that for $\alpha,\,\beta\in
H^1(J(X),\mathbb{Z})$,
\begin{eqnarray}\label{eqcup1}\sum_in_i\langle  \gamma_i^*\alpha,
\gamma_i^*\beta\rangle_{Z_i}=\langle
\frac{\theta^{g-1}}{(g-1)!}\alpha,\beta\rangle_{J(X)}.
\end{eqnarray}
Now, the right hand side is equal, by definition of the polarization
$\theta$, to $\langle  \alpha',\beta'\rangle_X$, where we use the
canonical isomorphism
$H^1(J(X),\mathbb{Z})\cong H^3(X,\mathbb{Z})$
to identify
$\alpha,\,\beta$ to classes $\alpha',\,\beta'$ of degree $3$ on $X$.
Finally, using again this canonical isomorphism,
we have :
$$\gamma_i^*\alpha=\Gamma_{i}^*\alpha',\,\,\gamma_i^*\beta=
\Gamma_{i}^*\beta'$$ so that the left hand side in (\ref{eqcup1}) is
equal to $\sum_in_i\langle  \Gamma_{i}^*\alpha',
\Gamma_{i}^*\beta'\rangle_{Z_i}$. Hence (\ref{eqcup1}) is equivalent
to the fact that for any $\alpha',\,\beta'\in H^3(X,\mathbb{Z})$,
$$\sum_in_i\langle  \Gamma_{i}^*\alpha',
\Gamma_{i}^*\beta'\rangle_{Z_i}=\langle \alpha',\beta'\rangle_{X},$$
which is equality (\ref{eqint10ju}).
\end{proof}
The following variant is proved as above, using Theorem  \ref{theodeccohdutextvariant} instead
of Theorem \ref{theodeccohdutext}. It answers a question asked to us by A. Beauville.
\begin{theo} \label{theovarN}Let $X$ be a rationally connected threefold with no torsion in
$H^3(X,\mathbb{Z})$. Assume that, for some integer $N$, $N\Delta_X$
admits a decomposition
\begin{eqnarray}
\label{eqdecompN3jul}[N\Delta_X]=N[X\times x]+[Z],
\end{eqnarray}
with $Z$ supported on $D\times X$, $D\subsetneqq X$. Then

(i)    There exists a codimension $2$ cycle $\Gamma\in {\rm
CH}^2(J(X)\times X)$ such that for any $t\in J(X)$,
$$\Phi_X(\Gamma_t)=N\, t\,\,{\rm in}\,\,J(X).$$

(ii) The integral cohomology class $N^2\theta^{g-1}/(g-1)!\in
H^{2g-2}(J(X),\mathbb{Z})$ is algebraic on $J(X)$, where $g={\rm
dim}\,J(X)$.

\end{theo}
\begin{proof} (i) Let $\widetilde{D}\stackrel{j}{\hookrightarrow} X$ be a desingularization with
a cycle $\widetilde{Z}\in {\rm CH}(\widetilde{D}\times X)$ such that
$(j,Id_X)_*(\widetilde{Z})=N\Delta_X$. It follows that for any
$\alpha\in H^3(X,\mathbb{Z})$,
$$N\alpha=j_*(\widetilde{Z}^*\alpha),$$
and similarly, looking at the induced morphisms of complex tori,
$$N \,Id_{J(X)}=j_*\circ \widetilde{Z}^*:J(X)\rightarrow J(X),$$
where $\widetilde{Z}^*$ here gives a morphism
$$\psi:J(X)\rightarrow J^1(\widetilde{D})\cong {\rm Pic}^0(\widetilde{D}).$$
 Now we use the
existence of a universal divisor $\mathcal{D}$ on ${\rm
Pic}^0(\widetilde{D})\times \widetilde{D} $. Let
$$\mathcal{Z}:=(Id_{J(X)},j)_*(\psi, Id_{\widetilde{D}})^*(\mathcal{D})\in {\rm CH}^2(J(X)\times X).$$
Then for any $t\in J(X)$,
$$\Phi_X( \mathcal{Z}_t)=j_*(\Phi_{\widetilde{D}}(\mathcal{D}_{\psi(t)}))\,\,{\rm in}\,\,J(X).$$
The right hand side is equal to
$j_*(\psi(t))=j_*(\widetilde{Z}^*(t))=N\,t$, proving (i).

(ii) We use Theorem \ref{theodeccohdutextvariant}. We thus have curves
$C_i$, correspondences $\Gamma_i\in {\rm CH}^2(C_i\times X)$ and integers
$n_i$ such that
for any $\alpha,\beta\in  H^3(X,\mathbb{Z})$,
$$N^2\langle  \alpha,\beta\rangle_X=\sum_in_i\langle  \Gamma_i^*\alpha,\Gamma_i^*\beta\rangle_{C_i}.$$
As in the proof of Theorem \ref{cororc3}, this equality exactly says
that the $1$-cycles $D_i:=\gamma_{i*}\in {\rm CH}_1( J(X))$, where
$\gamma_i=\Phi_X\circ \Gamma_{i*}:C_i\rightarrow J(X)$, satisfy
$$\sum_in_i [D_i]=N^2\frac{\theta^{g-1}}{(g-1)!}\,\,{\rm in}\,\,H^{2g-2}(J(X),\mathbb{Z}).$$

\end{proof}
\begin{rema}{\rm When $X$ is a unirational threefold admitting a degree $N$ unirational parametrization
$$\phi:\mathbb{P}^3\dashrightarrow X,$$
$N\Delta_X$ admits a decomposition as in (\ref{eqdecompN3jul}),
simply because, denoting $Y$ a blow-up of $\mathbb{P}^3$ on which
$\phi$ is desingularized to a true morphism $\tilde{\phi}$, one has
$$(\tilde{\phi},\tilde{\phi})_*(\Delta_Y)=N\Delta_X$$
and $Y$ admits a decomposition of the diagonal.
In this case, Theorem \ref{theovarN}, (ii) has an immediate proof which provides the following stronger statement:

}
There is an effective cycle of class $N\theta^{g-1}/(g-1)!$ in $H^{2g-2}(J(X),\mathbb{Z})$, where
$g={\rm dim}\,J(X)$.

{\rm To see this, recall from \cite{clemensgriffiths} that
$(J(Y),\theta_Y)$ is a direct sum of Jacobians of smooth curves.
Thus there exists a (possibly reducible) curve $C\subset J(Y)$ with
class $\theta_Y^{g'-1}/(g'-1)!$, where $g':={\rm dim}\,J(Y)$. Let
$\psi: J(Y)\rightarrow J(X)$ be the morphism induced by
$\tilde{\phi}:Y\rightarrow X$. We claim that $\psi_*(C)\subset J(X)$
has class $N \theta^{g-1}/(g-1)!$ in $J(X)$. Indeed, by definition
of the Theta divisor of $J(X)$, this is equivalent to  saying that
for any $\alpha,\,\beta\in H^3(X,\mathbb{Z})$, denoting by
$\alpha',\,\beta'$ the corresponding degree $1$ cohomology classes
on $J(X)$ via the isomorphism $H^3(X,\mathbb{Z})\cong
H^1(J(X),\mathbb{Z})$,
\begin{eqnarray}
\label{eqdermi} N\langle
\alpha,\beta\rangle_X=\int_{\psi_*(C)}\alpha'\wedge \beta'.
\end{eqnarray} However,
$$\int_{\psi_*(C)}\alpha'\wedge \beta'=\int_C{\psi}^*\alpha'\wedge  {\psi}^*\alpha',$$
where ${\psi}^*\alpha'$ identifies with $\phi^*\alpha\in
H^3(Y,\mathbb{Z})$ via the natural isomorphism
$H^1(J(Y),\mathbb{Z})\cong H^3(Y,\mathbb{Z})$, and similarly for
$\beta$. Finally, we get by definition of the Theta divisor of
$J(Y)$:
$$\int_C{\psi}^*\alpha'\wedge  {\psi}^*\alpha'=\langle  \tilde{\phi}^*\alpha,
\tilde{\phi}^*\beta\rangle_Y=N\langle \alpha,\beta\rangle_X,$$ which
proves (\ref{eqdermi}).

}
\end{rema}
In the case of a smooth  cubic threefold, we get the following
consequence of Theorem \ref{cororc3}:
\begin{coro} (cf. Theorem \ref{theocubthree}) A smooth cubic threefold admits a Chow-theoretic decomposition of
the diagonal (that is, its ${\rm CH}_0$ group is universally
trivial) if and only if the class $\theta^4/4!$ is algebraic on
$J(X)$.
\end{coro}
\begin{proof} Indeed, this is a necessary condition by Theorem \ref{cororc3}.
The fact that this is a sufficient condition is proved as follows:
the cubic $3$-fold has no torsion in $H^3(X,\mathbb{Z})$. It is not
known if it admits a universal codimension $2$ cycle, but it is
known by work of Markushevich-Tikhomirov \cite{martikho} that it
admits a parametrization of the intermediate Jacobian with
rationally connected fibers, that is, there exists a smooth projective variety
$B$, and a codimension $2$ cycle
$Z\in {\rm CH}^2(B\times X)$, such that the induced morphism
$$\Phi_Z:B\rightarrow J(X),$$
$$t\mapsto \Phi_X(Z_t)$$
is surjective with rationally connected general fiber. In
\cite[Theorem 4.1]{voisinjag}, it is proved that if such a
parametrization exists for a given rationally connected $3$-fold
$X$, and if furthermore the minimal class $\theta^{g-1}/(g-1)!$ is
algebraic on $J(X)$, then there exists a universal codimension $2$
cycle on $J(X)\times X$. In the case of  the cubic threefold, we
conclude that if the minimal class $\theta^4/4!$ is algebraic, then
there exists a universal codimension $2$ cycle on $J(X)\times X$.
Thus Theorem \ref{cororc3} implies that $X$ admits a cohomological
decomposition of the diagonal. By Theorem \ref{theoeqchcohdec}, $X$
admits then a Chow-theoretic decomposition of the diagonal.
\end{proof}
We conclude with the following result:
\begin{theo} There exists a  non-empty countable union  of proper subvarieties of codimension
$\leq 3$ in the moduli space of smooth cubic threefolds
parametrizing threefolds $X$ with universally trivial ${\rm CH}_0$
group.
\end{theo}
\begin{proof} We first claim  that if $(J(X),\theta)$ is isogenous  via an odd degree isogeny to $(J(C),m\theta_C)$
for some  (possibly reducible) curve $C$, then $J(X)$ has a
one-cycle whose class is an odd multiple of the minimal class
$\theta^4/4!$. Indeed,  we have the odd degree isogeny
$\mu:J(C)\rightarrow J(X)$ with the property that
$\mu^*\theta_X=m\theta_C$. As ${\rm deg}\,\mu$ is odd, $m$ is odd.
Furthermore, we have $\mu_*(\theta_C^4/4!)=m(\theta^4/4!)$ and
$\theta_C^4/4!$ is algebraic on $J(C)$. As $2(\theta^4/4!)$ is an
algebraic class because $(J(X),\theta)$ is a Prym variety (see
\cite{clemensgriffiths}), we conclude that $\theta^4/4!$ is
algebraic, which proves the claim.

An explicit example is as follows: Consider a smooth cubic threefold defined by a homogeneous polynomial
$P(X_0,\ldots,X_4)$, where $P$ is invariant under the automorphism
$g$ of order $3$ acting on
coordinates by
$$g^*X_0=X_0,\,g^*X_1=jX_1,\,g^*X_2=j^2X_2,\,g^*X_3=X_3,\,g^*X_4=X_4,$$
where $j={\rm exp}\,\frac{2\iota\pi}{3}$. The invariant part
$H^3(X,\mathbb{Q})^{inv}$ of $H^3(X,\mathbb{Q})$ under the action of
$g$ has rank $6$. This can be seen by looking at the action of $g^*$
on $H^{2,1}(X)$, the later space being computed via Griffiths
residues (see \cite[6.1]{voisinbook}): One gets a residue isomorphism
\begin{eqnarray}
\label{eqpourresidue} H^0(X,\mathcal{O}_X(1))\cong H^{2,1}(X) ,\,\,\,
 A\mapsto  {\rm Res}_X \frac{A\Omega}{P^2}, \end{eqnarray}
where $\Omega$ is the canonical generator of
$H^0(\mathbb{P}^4,K_{\mathbb{P}^4}(5))$. As $g^*\Omega=\Omega$,
(\ref{eqpourresidue}) induces an isomorphism
$$ H^0(X,\mathcal{O}_X(1))^{inv}\cong H^{2,1}(X)^{inv}$$
so that ${\rm dim}\,H^{2,1}(X)^{inv}=3$.

 Let $\pi=Id+g^*+(g^2)^*\in {\rm End}\,(H^3(X,\mathbb{Z}))$. Then
$\pi/3$ is the orthogonal projector of $H^3(X,\mathbb{Q})$ onto
$H^3(X,\mathbb{Q})^{inv}$
 with respect
to the intersection pairing on $H^3(X,\mathbb{Q})$.
Hence over $\mathbb{Q}$ we have an orthogonal decomposition
$$H^3(X,\mathbb{Q})=H^3(X,\mathbb{Q})^{inv}\oplus H^3(X,\mathbb{Q})^{\sharp},$$
where $H^3(X,\mathbb{Q})^{\sharp}={\rm Im}\,(Id-\pi/3)$.
Over $\mathbb{Z}$, we conclude that
$H^3(X,\mathbb{Z})$ contains a sublattice
$H^3(X,\mathbb{Z})^{inv}\oplus H^3(X,\mathbb{Z})^{\sharp}$, where
$$H^3(X,\mathbb{Z})^{\sharp}=H^3(X,\mathbb{Q})^{\sharp}\cap H^3(X,\mathbb{Z})=(H^3(X,\mathbb{Z})^{inv})^{\perp}.$$

The index of this sublattice is  a power of $3$, since
for $a\in H^3(X,\mathbb{Z})$, we can write
$$3a=(a+g^*a+(g^2)^*a)+(3a-(a+g^*a+(g^2)^*a)),$$
with
$$a+g^*a+(g^2)^*a\in H^3(X,\mathbb{Z})^{inv},\,3a-(a+g^*a+(g^2)^*a)\in H^3(X,\mathbb{Z})^{\sharp}.$$
It follows that  we can construct two finite index sublattices
$$H_1\subset H^3(X,\mathbb{Z})^{inv},\, H_2\subset H^3(X,\mathbb{Z})^{\sharp}$$
with the property that the restriction of the intersection form
$\langle ,\rangle_X$ to $H_1$ and $H_2$ is $m$ times a unimodular
intersection pairing, where $m$ is a power of $3$. These sublattices
determine principally polarized abelian varieties $A$ and $B$ of
respective dimensions $3$ and $2$, together with an isogeny
$$A\oplus B \rightarrow J(X)$$
such that the pull-back of $\theta_X$ is $m(\theta_A,\theta_B)$.
The ppav's $A$ and $B$ are Jacobians of curves
$C_A$, $C_B$, and
$(A\oplus B, (\theta_A,\theta_B))$ is the Jacobian of the
curve $C_A\cup_x C_B$.

We conclude with the following: \begin{lemm} Each choice of sublattices
$H_1,\,H_2$ as above provides us with a subvariety of codimension
$\leq3$ in the moduli space of $X$ along which the class
$\theta^4/4!$ is algebraic.
\end{lemm}
\begin{proof} Let $C=C_A\cup_x C_B$ be a curve as above, with $x$ general,
and let $$\mathcal{C}\rightarrow V,\,\,c\in V,\,\,\mathcal{C}_c\cong C$$ be a universal family of deformations of $C$. Similarly, denote by $U$ the base of a
universal family of deformations of $X$.
Denote by $\widetilde{A}_{5,X}$ the base of a universal
family of deformations of the ppav $J(X)$ and
$\widetilde{A}_{5,C}$ the base of a universal
family of deformations of the ppav $J(C)$.
We have an isogeny $\alpha:J(C)\stackrel{isog}{\cong} J(X)$
and the (local) period maps
$$\mathcal{P}_C:V\rightarrow \widetilde{\mathcal{A}}_{5,C},\,\,\mathcal{P}_X:U\rightarrow
\widetilde{\mathcal{A}}_{5,X}.$$
The isogeny $\alpha$ provides a local (for the Euclidean topology) isomorphism
$$\alpha:\widetilde{\mathcal{A}}_{5,C}\cong \widetilde{\mathcal{A}}_{5,X}.$$
As ${\rm
dim}\,\widetilde{\mathcal{A}}_{5,X}=15$, any component of  the subvariety $G\subset
V\times U$ defined by
$$G=\{(t,u)\in \overline{\mathcal{M}_5}\times
U,\,\alpha(\mathcal{P}_C(t))=\mathcal{P}_X(u)\}$$ has codimension $\leq 15$
in $V\times U$, hence has dimension $\geq 9$ since
${\rm dim}\,U={\rm dim}\,V={\rm dim}\,\overline{\mathcal{M}_5}=12$. The image $U'\subset U$ of
the second projection $p_2: G\rightarrow U$ consists of cubics whose
intermediate Jacobian is isogenous (via the given isogeny type) to a
Jacobian of curve. As ${\rm dim}\,U=12$, one
has ${\rm codim}\,U'\leq 3$, unless $p_2$ is not generically finite on
its image. In this case, the image of $G$ by the first projection is
contained in the locus of reducible curves, as this is the only
locus where the period map $\mathcal{P}_C$ has positive dimensional fibers.
 So we have to exclude this last possibility. Assume
by contradiction this happens. Note that the  intermediate Jacobian
of a generic cubic threefold $X$ with $\mathbb{Z}/3\mathbb{Z}$-action as above has only two
 simple factors, one of dimension $2$, the other of dimension $3$. Hence the fibers of the period map $\mathcal{P}_C$
over any isomorphism class of an abelian $5$-fold isogenous to $J(X)$ has dimension
 at most $2$.    As ${\rm dim}\,G\geq 9$, it follows that   $p_2(G)\subset U$ has dimension
$\geq7$, hence codimension $\leq5$, and is contained in the locus of $U$
parameterizing cubic threefolds with reducible intermediate
Jacobian. This can be excluded by an infinitesimal computation at
any point $x\in U$. In fact, it suffices  to prove that the
infinitesimal variation of Hodge structure at $x$
$$T_{U,x}\rightarrow {\rm Hom}\,(H^{2,1}(X), H^{1,2}(X))$$
maps  $T_{U,x}$ surjectively onto the $6$-dimensional space ${\rm
Hom}\,(H^{2,1}(X)^{inv}, H^{1,2}(X)^{\sharp})$, where
$H^{1,2}(X)^{\sharp}$ is defined as the orthogonal complement of
$H^{2,1}(X)^{inv}$. Indeed, the space $${\rm
Hom}\,(H^{2,1}(X)^{inv}, H^{1,2}(X)^{\sharp})$$ identifies with the
 normal bundle of the locus of reducible ppav's in $\widetilde{\mathcal{A}}_{5,X}$.
 Using Griffiths' theory (see \cite[6.1-2]{voisinbook}), this computation is easily performed  in the
Jacobian ring of  the Fermat  equation  $P=\sum_iX_i^3$.
\end{proof}
It is a standard fact that there are countably many choices of pairs
$(H_1,H_2)$ of lattices as above. Indeed, starting from the
unimodular intersection pairing on $H_1$, we can write $H_1$ as the
sum of two Lagrangian sublattices
$$H_1=\Lambda\oplus \Lambda',$$
and then for each integer $m$, we can consider
$H'_{1}(m,\Lambda,\Lambda'):=m\Lambda\oplus \Lambda'$; we  also have
to make a similar construction  for $H_2$.

  We  get this way countably many corresponding codimension $\leq3$
subvarieties and it is likely that they are Zariski dense in the
moduli space of cubic threefolds, but we did not try to prove this.
\end{proof}

\section{More results on cubic hypersurfaces
\label{sec5}}

Let $X$ be a smooth cubic hypersurface of dimension $n\geq3$. The
Hodge structure on $H^n(X,\mathbb{Q})_{prim}$ is a polarized
nontrivial Hodge structure (that is, when $n=2k+1$, it is nonzero,
and when $n=2k$, it is not purely of type $(k,k)$). Let ${\rm
End}_{HS}(H^n(X,\mathbb{Q})_{prim})$ be the space of  endomorphisms
of this Hodge structure.
\begin{lemm}\label{le27juin} For the very general cubic hypersurface of dimension $n$,
${\rm End}_{HS}(H^n(X,\mathbb{Q})_{prim})=\mathbb{Q}Id$.
\end{lemm}
\begin{proof} This follows  from the fact that the
Mumford-Tate group of the considered Hodge structures is the
symplectic group if $n$ is odd and the orthogonal group if $n$ is
even. This fact in turn follows from the fact that the Mumford-Tate
group contains a finite index subgroup of the monodromy group (see
\cite{voisinhodgeloci}) and that by \cite{beauville}, the monodromy
group of a smooth hypersurface is Zariski dense in  the symplectic
or orthogonal group of $H^n(X,\mathbb{Q})_{prim}$ except for even
dimensional quadrics and  cubic surfaces, for which it is finite.
This immediately implies the lemma because by definition of the
Mumford-Tate group $G:={\rm MT}(H^n(X,\mathbb{Q})_{prim})$, any
endomorphism $\phi$ of the Hodge structure on
$H^n(X,\mathbb{Q})_{prim}$ has to commute with $G$, that is
$\phi\circ g=g\circ \phi$ for $g\in G$.
\end{proof}
 We now have the following result, saying that for
cubic hypersurfaces satisfying the conclusion of Lemma
\ref{le27juin}, the ${\rm CH}_0$  group cannot be universally
supported on a proper closed algebraic subset of $X$, unless it is
trivial.
 Let $Y\subset X$ be a proper closed algebraic subset. We introduced in Definition \ref{defiuniv}
 the notion of
 ${\rm CH}_0(Y)\rightarrow {\rm CH}_0(X)$ being universally surjective.
When $Y$ is a point, the fact that ${\rm CH}_0(Y)\rightarrow {\rm
CH}_0(X)$ is universally surjective is equivalent to the fact that
$X$ has  universally trivial ${\rm CH}_0$ group.
\begin{theo} \label{theochowinterdim}  Let $X$ be a smooth cubic hypersurface such that
$H^n(X,\mathbb{Z})/H^n(X,\mathbb{Z})_{alg}$ has no $2$-torsion for
 $n={\rm dim}\,X$, and ${\rm
End}_{HS}(H^n(X,\mathbb{Q})_{prim})=\mathbb{Q}Id$. Assume there is a
proper closed algebraic subset $Y\subset X$ such that ${\rm
CH}_0(Y)\rightarrow {\rm CH}_0(X)$ is universally surjective,
 Then ${\rm CH}_0(X)$ is universally
trivial.
\end{theo}
The assumptions of the Theorem are satisfied by a very general cubic hypersurface, which proves
Theorem \ref{theonewintro} stated in the introduction.

\begin{proof}[Proof of Theorem \ref{theochowinterdim}] Let $L=\mathbb{C}(X)$. Then we have the diagonal point
$\delta_L$ and the fact that it comes from a $0$-cycle supported on
$Y_L$ says,  by taking the Zariski closure in $X\times X$ and using
the localization exact sequence, that there is a decomposition of
the diagonal of $X$ which takes the following form:
$$\Delta_X=Z_1+Z_2\,\,{\rm in}\,\,{\rm CH}^n(X\times X),
$$
where $Z_1$ is supported on $D\times X$ for some proper closed
algebraic subset $D\subset X$, and $Z_2$ is supported on $ X\times
Y$. This decomposition gives in particular a cohomological
decomposition:
\begin{eqnarray}
\label{eqdecompcoaecY}[\Delta_X]=[Z_1]+[Z_2]\,\,{\rm
in}\,\,H^{2n}(X\times X,\mathbb{Z}),
\end{eqnarray}
 where $Z_1$ and $Z_2$ are as above. We now use Lemma
 \ref{newlemma27juin} below which says that
 a decomposition as in (\ref{eqdecompcoaecY}) implies that
 $X$ admits a cohomological decomposition of the diagonal because we
 assumed
 ${\rm End}_{HS}(H^n(X,\mathbb{Q})_{prim})=\mathbb{Q}Id$,
  and Theorem \ref{theoeqchcohdec} which says that $X$ admits then a
  Chow-theoretic decomposition of the diagonal because we assumed that $H^n(X,\mathbb{Z})/H^n(X,\mathbb{Z})_{alg}$
  has no $2$-torsion; hence we proved that
${\rm CH}_0(X)$ is in fact universally trivial.
\end{proof}
\begin{lemm}\label{newlemma27juin} Let $X$ be a smooth cubic
hypersurface. Assume that  there is a decomposition
\begin{eqnarray}
\label{eqdecompcoaecY2}[\Delta_X]=[Z_1]+[Z_2]\,\,{\rm
in}\,\,H^{2n}(X\times X,\mathbb{Z}),
\end{eqnarray}
where $Z_1$ is supported on $D\times X$ , and $Z_2$ is supported on
$ X\times Y$ for some proper closed algebraic subsets $D,\,Y\subset
X$. Then if furthermore ${\rm
End}_{HS}(H^n(X,\mathbb{Q})_{prim})=\mathbb{Q}Id$, $X$ admits a
cohomological decomposition of the diagonal.
\end{lemm}
\begin{proof} Indeed, recall that cubic hypersurfaces admit a
unirational parametrization of degree $2$. So $2[\Delta_X]$ can be
decomposed as $2[X\times x]+[Z]$, where $Z$ is supported on $D\times
X$ for a proper closed algebraic subset $D\subset X$. Hence it
suffices to show that for some odd integer $m$, we
 have a decomposition
$$m[\Delta_X]=m [X\times x]+[Z]\,\,{\rm in}\,\,H^{2n}(X\times X,\mathbb{Z}),$$ where $Z\in {\rm CH}^n(X)$ is supported on $D\times
X$ for some closed algebraic subset $D\subsetneqq X$. Consider the
decomposition (\ref{eqdecompcoaecY2}): each class appearing in this
decomposition acts on $H^n(X,\mathbb{Z})_{prim}$ via a morphism of
Hodge structures, the diagonal acting as identity. As ${\rm
End}_{HS}(H^n(X,\mathbb{Z})_{prim})=\mathbb{Z}Id$, we have
$$[Z_1]^*=m_1 Id,\,[Z_2]^*=m_2 Id,$$
where $m_1,\,m_2$ are two integers such that  $m_1+m_2=1$. We may assume that $m_1$ is odd,  applying
transposition to our cycles if necessary. It follows that $m_2$ is even,
 and  the cycle class $[Z_2]-m_2[\Delta_X]$ acts
trivially on $H^n(X,\mathbb{Q})_{prim}$. Over $\mathbb{Q}$, using
the orthogonal decomposition
$$H^*(X,\mathbb{Q})=H^n(X,\mathbb{Q})_{prim}\oplus
H^*(\mathbb{P}^{n+1},\mathbb{Q})_{\mid X},$$ we conclude that for
some rational numbers $\alpha_i$
\begin{eqnarray}
\label{eqdecompcoaecY3}[Z_2]-m_2[\Delta_X]=\sum_i\alpha_i
h_1^{i}\otimes h_2^{n-i}\,\,{\rm in}\,\,H^{2n}(X\times
X,\mathbb{Q}),
\end{eqnarray}
where here $h\in H^2(X,\mathbb{Z})$ is $c_1(\mathcal{O}_X(1))$ and
the right hand side makes sense in $H^{2n}(X\times X,\mathbb{Q})$
via K\"unneth decomposition. We observe now that, because
$H^*(X,\mathbb{Z})$ has no torsion and the pairings $\langle
h^i,h^{n-i}\rangle$ are equal to $3$, the denominators in the
coefficients $\alpha_i$ of (\ref{eqdecompcoaecY3}) are equal to $3$
(or $1$), so that we get
\begin{eqnarray}
\label{eqdecompcoaecY4}3[Z_2]-3m_2[\Delta_X]=\sum_i\beta_i
pr_1^*h^{i}\smile pr_2^*h^{n-i}\,\,{\rm in}\,\,H^{2n}(X\times
X,\mathbb{Z}),
\end{eqnarray}
where now the $\beta_i$ are integers. Combining
(\ref{eqdecompcoaecY4}) with (\ref{eqdecompcoaecY2}), we get now:
\begin{eqnarray}
\label{eqdecompcoaecY5}3[\Delta_X]=3[Z_1]+3m_2[\Delta_X]+\sum_i\beta_i
pr_1^*h^{i}\smile pr_2^*h^{n-i}\,\,{\rm in}\,\,H^{2n}(X\times
X,\mathbb{Z}),
\end{eqnarray}
where $Z_1$ is  supported on $D\times X$  for some  closed
algebraic subset $D\subsetneqq X$. Hence we proved that
\begin{eqnarray}
\label{eqdecompcoaecY6}(3-3m_2)[\Delta_X]-3\beta_0[X\times
x]=3[Z_1]+\sum_{i>0}\beta_i pr_1^*h^{i}\smile pr_2^*h^{n-i}\,\,{\rm
in}\,\,H^{2n}(X\times X,\mathbb{Z}),
\end{eqnarray}
and clearly, the class $\sum_{i>0}\beta_i pr_1^*h^{i}\smile
pr_2^*h^{n-i}$ is the class of a cycle supported on $D'\times X$,
for some  closed algebraic subset $D'\subsetneqq X$. As $m_2$ is
even, $3-3m_2$ is odd, and thus (\ref{eqdecompcoaecY6}) finishes the
proof.
\end{proof}
In the case of cubic fourfolds, we can replace in Theorem
\ref{theochowinterdim} the assumption that ${\rm
End}_{HS}(H^n(X,\mathbb{Q})_{prim})=\mathbb{Q}Id$ by the following
assumption which concerns the {\it simple} Hodge structure
$$H^4(X,\mathbb{Q})_{tr}:=H^4(X,\mathbb{Q})_{alg}^{\perp},$$
namely
 ${\rm End}_{HS}(H^4(X,\mathbb{Q})_{tr})=\mathbb{Q}Id$. This is
 a weaker and more natural assumption because, when the Hodge structure on
 $H^4(X,\mathbb{Q})_{prim}$ is not simple, or equivalently,
 when
 $H^4(X,\mathbb{Q})_{prim}$ contains nonzero Hodge classes,
 the algebra
 ${\rm
End}_{HS}(H^4(X,\mathbb{Q})_{prim})$ contains projectors associated
to sub-Hodge structures, and this happens along codimension $1$ loci
paramerizing special cubic fourfolds. On the contrary, the
non-existence of nontrivial endomorphisms of the Hodge structure on
$H^4(X,\mathbb{Q})_{tr}$ is satisfied in codimension $1$, and in
particular at the very general point of a Noether-Lefschetz locus by
the following lemma:
\begin{lemm}\label{ledim10} In the moduli space of smooth cubic fourfolds, the set of points
parameterizing cubics $X$ such that ${\rm
End}_{HS}(H^4(X,\mathbb{Q})_{tr})\not=\mathbb{Q}Id$ is of
codimension $\geq2$.
\end{lemm}
\begin{proof} Equivalently, we have to show that this set does not
contain the very general point of a Noether-Lefschetz locus
$\mathcal{D}_\sigma$ defined by a class $\sigma$, or the very
general point of the moduli space. By contradiction, let $X$ be  a
very general point of $\mathcal{D}_\sigma$, and let $h\in {\rm
End}_{HS}(H^4(X,\mathbb{Q})_{tr})$ be a morphism of Hodge structures
(which then has to remain a morphism of Hodge structures acting on
$H^4(X_t,\mathbb{Q})_{tr}$ for any small deformation $X_t$ of $X$
parameterized by a point $t\in \mathcal{D}_\sigma$) but is not a
homothety. Let $\lambda\in \mathbb{C}$ be the algebraic number such
that $h^*\eta_X=\lambda \eta_X$, where $\eta_X$ is a generator of
the rank $1$ vector space $H^{3,1}(X)$. As the Hodge structure on
$H^4(X,\mathbb{Q})_{tr}$ is simple and $h$ is not a homothety,
$\lambda$ is not a rational number. It follows that the eigenspace
$H_\lambda$ of $h$ associated with the eigenvalue $\lambda$ has
complex dimension $\leq \frac{1}{2}{\rm
dim}\,H^4(X,\mathbb{Q})_{tr}$. On the other hand, the period map
restricted to $\mathcal{D}_\sigma$ has by assumption its image
contained in $\mathbb{P}(H_\lambda)$. As the period map is
injective, we conclude that $${\rm dim}\,\mathcal{D}_\sigma\leq
\frac{1}{2}({\rm dim}\,H^4(X,\mathbb{Q})_{tr}-2)$$ which is absurd since
the right-hand side is equal to $19/2$ while the left-hand side is
equal to $19$. This contradicts our assumption that $h$ is not a
homothety. The same argument works if $X$ is a general point of the
moduli space.
\end{proof}
Our next result is the following:

\begin{theo} \label{theochowinterdimvariant}  Let $X$ be a smooth cubic fourfold such that
${\rm End}_{HS}(H^4(X,\mathbb{Q})_{tr})=\mathbb{Q}Id$. Assume there
is a proper closed algebraic subset $Y\subset X$ such that ${\rm
CH}_0(Y)\rightarrow {\rm CH}_0(X)$ is universally surjective. Then ${\rm CH}_0(X)$ is universally
trivial.
\end{theo}
\begin{proof} The proof is very similar to the previous proof. The assumption is that
$$[\Delta_X]=[Z_1]+[Z_2]\,\,{\rm in}\,\,H^{8}(X\times X,\mathbb{Z}),$$
with $Z_1$ supported on $D_1\times X$, $Z_2$ supported on $X\times
D_2$ for some closed proper algebraic subsets $D_1,\,D_2\subsetneqq
X$. We consider the action of $[Z_i]^*$ on $H^4(X,\mathbb{Z})_{tr}$.
As ${\rm End}_{HS}(H^4(X,\mathbb{Q})_{tr})=\mathbb{Q}Id$, each of
them must act as a multiple of the identity and the sum
$[Z_1]^*+[Z_2]^*$ equals $Id_{H^4(X,\mathbb{Z})_{tr}}$. So one of
them, say $[Z_1]^*$, must act as an odd multiple of the identity and
the other as an even multiple of the identity. Let $[Z_2]^*=2m\,Id$
on $H^4(X,\mathbb{Z})_{tr}$. Then $(2m[\Delta_X]-[Z_2])^*$ acts as
$0$ on $H^4(X,\mathbb{Z})_{tr}$. Note also that
$(2m[\Delta_X]-[Z_2])^*$ maps
${Hdg}^4(X,\mathbb{Z})=H^*(X,\mathbb{Z})_{tr}^{\perp}$ to itself,
and this implies that
$$2m[\Delta_X]-[Z_2]\in {Hdg}^4(X,\mathbb{Z})\otimes {Hdg}^4(X,\mathbb{Z})\oplus \sum_{i\not= 2}H^{2i}(X,\mathbb{Z})
\otimes H^{8-2i}(X,\mathbb{Z})\subset  H^8(X\times X,\mathbb{Z}).$$
All classes in ${Hdg}^4(X,\mathbb{Z})$  are algebraic by the Hodge
conjecture for integral Hodge classes on cubic fourfolds proved in
\cite{voisinaspects} and all classes in $H^{*\not=4}(X,\mathbb{Z})$
are algebraic. Hence we conclude that
$$ 2m[\Delta_X]-[Z_2]= \sum_ipr_1^*[W_i]\smile pr_2^* [W'_i]$$
for some integral cycles $W_i,\,W'_i$ on $X$ satisfying ${\rm
dim}\,W_i+{\rm dim}\,W'_i=4$. The rest of the proof works as before,
allowing to conclude that $X$ admits a cohomological decomposition
of the diagonal, hence also a Chow-theoretic one by Theorem
\ref{theoeqchcohdec}.
\end{proof}
We finally prove the following Theorem \ref{thecbufin}.
 Let $X$ be a cubic fourfold. Assume $X$ is  special in the sense of Hassett, that is
$H^4(X,\mathbb{Z})$ contains two independent Hodge classes (one
being the class $h^2$, the other being denoted $\sigma$). Let
$P\subset H^4(X,\mathbb{Z})$ be the sublattice generated by these
two Hodge classes. The restriction of the intersection form $\langle\,
,\,\rangle_X$ to $P$ has a discriminant $D(\sigma)$. This number,
which is always even, has been very much studied in conjunction with
rationality properties of cubic fourfolds (see \cite{hassett},
\cite{hassettthschinkel}, \cite{kuznetsov}, \cite{addi}). Hassett's
work suggested  that if a cubic fourfold is rational, then it is
special and the discriminant $D(\sigma)$ satisfies severe
restrictions.

\begin{theo}\label{thecbufin} If $4$ does not divide $D(\sigma)$, the ${\rm CH}_0$ group of $X$ is universally
trivial (that is, $X$ admits a Chow-theoretic decomposition of the
diagonal).
\end{theo}
\begin{proof}   The assumption on $X$
being satisfied along a countable union of Noether-Lefschetz type
divisors $\mathcal{D}_\sigma$ in the moduli space of cubic fourfolds
 (see \cite{hassett}), it suffices to show that the conclusion
holds for $X$ very general in each $\mathcal{D}_\sigma$. Indeed,
inside each  divisor $\mathcal{D}_\sigma$, the existence of a
cohomological (or Chow-theoretic) decomposition of the diagonal is
satisfied  along a countable union of closed algebraic subsets (see
\cite[Proof of Theorem 1.1]{voisindoublesolid}). So if it satisfied
at the very general point of $\mathcal{D}_\sigma$, it is satisfied
everywhere along $\mathcal{D}_\sigma$. Next, by Lemma \ref{ledim10},
 for a very general  point $X$ in a Noether-Lefschetz
  divisor, we have
${\rm End}_{HS}(H^4(X,\mathbb{Q})_{tr})=\mathbb{Q}Id$. Theorem
\ref{theochowinterdimvariant} thus tells us that if there is a surface
$\Sigma\subset X$ such that ${\rm CH}_0(\Sigma)\rightarrow {\rm CH}_0(X)$ is universally surjective,
then ${\rm CH}_0(X)$ is universally trivial.
The existence of a special Hodge class $\sigma$ provides us with an algebraic cycle
$Z$ on $X$ of class $\sigma$, by the Hodge conjecture for integral Hodge classes proved in
\cite{voisinaspects}.
Adding to $\sigma$ a high multiple of $h^2$, we can even assume that $Z$ is the class
of a smooth surface $\Sigma$ in general position. Indeed, as we are working with codimension $2$ cycles,
their classes are generated by $c_2$ of vector bundles on $X$. For any vector bundle
$E$ of rank $r$, a twist of $E$ is very ample and its $c_2$ is represented by a rank locus
associated to a morphism $\phi: \mathcal{O}^{r-1}\rightarrow E$, and this rank locus is smooth
in codimension $5$.  Let now
$\Sigma$  be as above. Recall the rational map
$$\Phi':X^2\dashrightarrow X,\,\,\Phi'(x,x')=x'',\,\,x+x'+x''=\langle x,x'\rangle\cap X.$$

We now have
\begin{lemm}\label{lesurfdeg}  Assume the restriction of $\Phi'$ to $\Sigma^2$ is dominant of
 degree $2N$ not divisible by $4$. Then the map ${\rm CH}_0(\Sigma)\rightarrow {\rm CH}_0(X)$ is universally surjective.
\end{lemm}
\begin{proof} Indeed, the rational  map $\Phi'_{\mid \Sigma\times \Sigma}$ is symmetric, that is
factors through a rational map $\psi:\Sigma^{(2)}\dashrightarrow X$,
and our assumption implies that the factored map $\psi$ has odd
degree $N$. It follows that $\psi: {\rm
CH}_0(\Sigma^{(2)}_L)\rightarrow {\rm CH}_0(X_L)$ is surjective for
any field $L$ containing $\mathbb{C}$ because its cokernel is
annihilated by $N$, with $N$ odd,  and also by $2$, since $X$ is a
cubic of dimension $\geq2$, hence admits a unirational
parametrization of degree $2$. On the other hand, if $z\in {\rm
CH}_0(\Sigma^{(2)}_L)$ is of degree $k$, $z$ provides a $0$-cycle
$z'$ on $\Sigma_L$ of degree $2k$, and we obviously have
$$\psi_*(z)=k\, h^4-j_*(z')\,\,{\rm in}\,\,{\rm CH}_0(X_L),$$
where $j:\Sigma\rightarrow X$ is the inclusion map. It follows that
the map $j_*:{\rm CH}_0(\Sigma_L)\rightarrow {\rm CH}_0(X_L)$ is
surjective as well, since the class $h^4 \in {\rm CH}_0(X_L)$
belongs to the image of $j_*$. Indeed, it belongs to the image of
${\rm CH}_0(X_\mathbb{C})=\mathbb{Z}x_0 \rightarrow{\rm CH}_0(X_L)$
for any point $x_0\in X(\mathbb{C})$, which one can take in
$\Sigma(\mathbb{C})$.
\end{proof}
The next lemma relates the degree of $\Phi'_{\mid \Sigma\times
\Sigma}$ to the discriminant $D(\sigma)$.
\begin{lemm}\label{propcalculdegre} Let $\Sigma\subset X$ be a smooth surface in general position. Then
the degree of the rational  map $\Phi'_{\mid \Sigma\times
\Sigma}:\Sigma\times \Sigma\dashrightarrow X$ is congruent to
$D(\sigma)$ modulo $4$, where $\sigma=[\Sigma]$.
\end{lemm}
\begin{proof} Let $x\in X$ be a general point of $X$ and let
$$\pi_x: X\dashrightarrow \mathbb{P}^4$$
be the linear projection from $x$. To say that $(z,z')\in \Sigma^2$
satisfies $\Phi'(z,z')=x$ is equivalent to say that $z,z'$ and $x$
are collinear, or that $\pi_x(z)=\pi_x(z')$. As $\Sigma$ is in
general position,the restriction of  $\pi_x$ to $\Sigma$ maps
$\Sigma$ to a surface $\Sigma'$ which is smooth apart from finitely
many double points corresponding to pairs $\{z,z'\}$ as above. It
follows that the degree of $\Phi'_{\mid \Sigma\times \Sigma}$ is
equal to twice the number $N$ of these double points (this argument
appears in \cite[7.2]{hassettthschinkel}). We now compare the
geometry of the two immersions
$$\Sigma\subset X,\,\pi_{x,\Sigma}:=\pi_{x\mid \Sigma}:\Sigma\rightarrow \mathbb{P}^4.$$
The two corresponding normal bundle exact sequences give
\begin{eqnarray}
\label{eqnormalbundle}
0\rightarrow T_\Sigma\rightarrow T_{X\mid \Sigma} \rightarrow N_{\Sigma/X}\rightarrow 0,
\\
\nonumber
0\rightarrow T_\Sigma\rightarrow \pi_{x,\Sigma}^*T_{\mathbb{P}^4} \rightarrow N_{\Sigma/\mathbb{P}^4}\rightarrow 0,
\end{eqnarray}
which by Whitney formula provides
\begin{eqnarray}
\label{eqnormalbundlewhitney}
c_2(T_\Sigma)=c_2( T_{X\mid \Sigma})-c_2(  N_{\Sigma/X})+K_\Sigma\cdot c_1( N_{\Sigma/X}),
\\
\nonumber
c_2( T_\Sigma)=c_2(\pi_{x,\Sigma}^*T_{\mathbb{P}^4})-c_2(N_{\Sigma/\mathbb{P}^4})+K_\Sigma\cdot
c_1( N_{\Sigma/\mathbb{P}^4}).
\end{eqnarray}
We now use the equalities
$$c_2(T_{X\mid \Sigma})=6h^2_{\mid \Sigma},\,c_2(\pi_{x,\Sigma}^* T_{\mathbb{P}^4})=10h^2_{\mid \Sigma},$$
$$c_1( N_{\Sigma/X})=K_\Sigma+3h_{\mid \Sigma},\,c_1( N_{\Sigma/\mathbb{P}^4})=K_\Sigma+5h_{\mid \Sigma}$$
together with
$$c_2(  N_{\Sigma/X})=\sigma^2,\,\sigma=[\Sigma]\in H^4(X,\mathbb{Z}),\,$$
$$ c_2(  N_{\Sigma/\mathbb{P}^4})=(h^2_{\mid \Sigma})^2-2N=(\sigma\cdot h^2)^2-2N.$$
Thus (\ref{eqnormalbundlewhitney}) becomes
\begin{eqnarray}
\label{eqnormalbundlewhitney2}
c_2( T_\Sigma)=6h^2\cdot\sigma-\sigma^2+K_\Sigma\cdot (K_\Sigma+3h),
\\
\nonumber
c_2( T_\Sigma)=10h^2\cdot\sigma-(\sigma\cdot h^2)^2+2N+K_\Sigma\cdot (K_\Sigma+5h).
\end{eqnarray}
We now add these two equalities and consider the result modulo $4$, which gives
\begin{eqnarray}
\label{eqnormalbundlewhitney3}
2c_2( T_\Sigma)=-\sigma^2-(\sigma\cdot h^2)^2+2N+2K_\Sigma^2 \,\,{\rm mod}\,\,4.
\end{eqnarray}
As $2c_2( T_\Sigma)-2K_\Sigma^2$ is divisible by $4$ by Noether's formula,
we conclude that
\begin{eqnarray}
\label{eqnormalbundlewhitney4}
\sigma^2+(\sigma\cdot h^2)^2=2N \,\,{\rm mod}\,\,4.
\end{eqnarray}
As $D(\sigma)=3\sigma^2-(\sigma\cdot h^2)^2$ is equal to $-\sigma^2-(\sigma\cdot h^2)^2$ modulo $4$,
we proved that
$2N=-D(\sigma)=D(\sigma)\,\,{\rm mod}\,\,4$.
\end{proof}
Combining Lemmas \ref{propcalculdegre} and \ref{lesurfdeg}, we conclude that
the map ${\rm CH}_0(\Sigma)\rightarrow {\rm CH}_0(X)$ is universally surjective, hence that
${\rm CH}_0(X)$ is universally trivial.
\end{proof}
\begin{rema}{\rm If one looks at the proof of the integral Hodge
conjecture for cubic fourfolds given in \cite{voisinaspects}, one
easily sees that it gives more, namely: the group of Hodge classes
of degree $4$ on a cubic fourfold is generated by classes of
rational surfaces. Thus the surface $\Sigma$ above can be chosen
rational. However, Lemma \ref{propcalculdegre} does not allow us to
conclude that if $D(\sigma)$ is not divisible by $4$, $X$ admits a
unirational parametrization of odd degree. Indeed, the rational
surface produced by the construction of \cite{voisinaspects} will be
presumably singular, and not in general position.}
\end{rema}

4  place Jussieu, 75252 Paris Cedex 05, France

\smallskip
 claire.voisin@imj-prg.fr
    \end{document}